\documentclass[11pt]{article}
\usepackage[utf8]{inputenc}
\usepackage{amsmath,amssymb}
\usepackage{url}
\usepackage[ruled,commentsnumbered,linesnumbered,vlined]{algorithm2e}
\usepackage{graphicx}
\usepackage{subfigure}
\usepackage{multirow}
\usepackage{amsthm}
\usepackage{indentfirst}
\usepackage{soul}
\usepackage{float}
\usepackage{verbatim}
\usepackage{calligra}
\usepackage[T1]{fontenc}
\usepackage{fullpage}
\usepackage[shortlabels]{enumitem}
\usepackage[title]{appendix}
\usepackage{caption}
\usepackage{natbib}
\bibpunct{(}{)}{;}{a}{,}{;}

\usepackage{mathtools} 

\usepackage{tikz, color}

\usepackage{xr}
\makeatletter
\newcommand*{\addFileDependency}[1]{
  \typeout{(#1)}
  \@addtofilelist{#1}
  \IfFileExists{#1}{}{\typeout{No file #1.}}
}
\makeatother

\newtheorem{proposition}{Proposition}
\newtheorem{lemma}[proposition]{Lemma}

\newtheorem{corollary}{Corollary}

\DeclareRobustCommand{\rchi}{{\mathpalette\irchi\relax}}
\newcommand{\irchi}[2]{\raisebox{\depth}{$#1\chi$}} 

\newcommand{\rafaelC}[1]{\textcolor{blue}{#1}}

\newcommand{\MateusF}[1]{\footnote{\textcolor{orange}{{\sc mateus}: #1}}}

\title{The connected Grundy coloring problem: Formulations and a local-search enhanced biased random-key genetic algorithm}

\author{
Mateus Carvalho{\thanks{Institute of Computing, Universidade Federal da Bahia, Salvador, BA 40170-115, Brazil ({\tt mateuscs@ufba.br}).}}
\and
Rafael A. Melo{\thanks{Institute of Computing, Universidade Federal da Bahia, Salvador, BA 40170-115, Brazil ({\tt rafael.melo@ufba.br}).}}
\and
Mauricio G. C. Resende {\thanks{Industrial \& Systems Engineering, University of Washington, 3900 E Stevens Way NE, Seattle, WA 98195, USA ({\tt mgcr@uw.edu}).}}
\and
Marcio C. Santos {\thanks{Departamento de Ciência da Computação, Universidade Federal de Minas Gerais, Belo Horizonte, MG 31270‐010, Brazil ({\tt marciocs@dcc.ufmg.br}).}}
\and 
Rodrigo F. Toso \thanks{Microsoft Corporation, One Microsoft Way, Redmond, WA 98052, USA ({\tt rofran@microsoft.com}).}
}

\begin{document}

\maketitle

\begin{abstract}

Given a graph $G=(V,E)$, a connected Grundy coloring is a proper vertex coloring that can be obtained by a first-fit heuristic on a connected vertex sequence. A first-fit coloring heuristic is one that attributes to each vertex in a sequence the lowest-index color not used for its preceding neighbors. A connected vertex sequence is one in which each element, except for the first one, is connected to at least one element preceding it. The connected Grundy coloring problem consists of obtaining a connected Grundy coloring maximizing the number of colors.
In this paper, we propose two integer programming (IP) formulations and a local-search enhanced biased random-key genetic algorithm (BRKGA) for the connected Grundy coloring problem. 
The first formulation follows the standard way of partitioning the vertices into color classes while the second one relies on the idea of representatives in an attempt to break symmetries. 
The BRKGA encompasses a local search procedure using a newly proposed neighborhood. A theoretical neighborhood analysis is also presented. 
Extensive computational experiments indicate that the problem is computationally demanding for the proposed IP formulations. Nonetheless, the formulation by representatives outperforms the standard one for the considered benchmark instances. Additionally, our BRKGA can find high-quality solutions in low computational times for considerably large instances, showing improved performance when enhanced with local search and a reset mechanism. Moreover we show that our BRKGA can be easily extended to successfully tackle the Grundy coloring problem, i.e., the one without the connectivity requirements.
\\

\end{abstract}

\section{Introduction}
\label{sec:introduction}

Graph coloring finds applications in several areas. A few recent examples include communication networks~\citep{ZhuDaiWan15}, video synopsis~\citep{HeGaoSanQuHan17}, and railway station design~\citep{JovPavBelMil20}. First-fit coloring heuristics, especially the connected heuristics (e.g., based on breadth-first search), are widely used to obtain solutions to these problems. Therefore, the Grundy number and the connected Grundy number have significant applicability, as they provide a theoretical worst-case measure of solution quality. It is noteworthy that the connected Grundy coloring problem, whose optimal solution provides the connected Grundy number, is NP-hard~\citep{BenCamDouGriMorSamSil14}.

\subsection{Basic definitions}
\label{sec:basicdefinitions}

Given a graph $G=(V,E)$ and a set of colors $K$, a \textit{vertex coloring}, or simply \textit{coloring}, is a mapping $c:V\rightarrow K$. A \textit{proper coloring} is one in which $c(u) \neq c(v)$ for every $uv\in E$. A \textit{$k$-coloring} is a coloring using exactly $k$ colors, i.e.,  one in which the image of $V$ under $c$ has $k$ elements. 
From now on, consider $|V| = n$ and $|E| = m$.
The \textit{chromatic number} of a graph $G$, $\rchi(G)$, is the minimum $k$ such that $G$ admits a proper $k$-coloring. For simplicity, in the remainder of the paper, a coloring is defined as proper unless stated otherwise.
Besides, consider the \textit{neighborhood} of $v$, $N(v)$, as the set of vertices adjacent to~$v$ and its \textit{anti-neighborhood}, $\bar{N}(v)$, as the set of vertices not adjacent to $v$. Additionally, let ${N}[v] = {N}(v) \cup \{v\}$ be the \textit{closed neighborhood}, and $\bar{N}[v] = \bar{N}(v) \cup \{v\}$ be the \textit{closed anti-neighborhood} of $v$. Let $d(v) = |N(v)|$ be the degree of the vertex $v$ and $\Delta(G)$ be the largest degree in $G$. The \textit{density} of a graph is defined as $dens(G)=\frac{2\times|E|}{|V|\times(|V|-1)}$.

Given a sequence $\sigma = (v_{1}, v_{2}, \ldots, v_{n})$ of the vertices, the \textit{first-fit} coloring heuristic assigns to each vertex $v_i$ the lowest-index color not used for its neighbors in $(v_1, \ldots, v_{i-1})$. A \textit{Grundy coloring} is a coloring that respects the properties of the first-fit heuristic, i.e., if $c(v) = k$ then there is some neighbor $u\in N(v)$ such that $c(u)= k'$ for every $k'<k$. 
The \textit{Grundy number}, $\Gamma(G)$, is the largest $k$ such that $G$ admits a Grundy $k$-coloring. The \textit{Grundy coloring problem} consists of obtaining a Grundy coloring maximizing the number of colors.

A connected sequence $\sigma_c = (v_1, \ldots, v_n)$ is a sequence of vertices ensuring that $v_i$, for $1 \leq i \leq n$, has at least one neighbor in the interval $(v_1,\ldots,v_ {i-1})$. A \textit{connected Grundy coloring} is a coloring that can be obtained over a connected sequence of vertices respecting the properties of the first-fit heuristic. 
The \textit{connected Grundy coloring problem} consists of obtaining a connected Grundy coloring maximizing the number of colors. The optimal solution to the problem is the \textit{connected Grundy number}, $\Gamma_c(G)$. The \textit{connected chromatic number}, $\rchi_c(G)$, denotes the smallest $k$ such that $G$ accepts a $k$-connected coloring that respects the first-fit heuristic property. Note that $\rchi(G) \leq \Gamma_c(G) \leq \Gamma(G) \leq \Delta(G) + 1$.

Figure \ref{fig:con_coloring} illustrates two examples of Grundy coloring, demonstrating the distinction between non-connected and connected colorings. Figure~\ref{fig:n_grundy} illustrates a non-connected Grundy coloring. 
Notice that there is no manner to achieve this coloring in a connected way, because for the vertex $a$ to receive color 4 and $b$ to receive color 3, the vertices in $\{d,c,e,f\}$ need to be colored first and this subset does not induce a connected subgraph. Figure \ref{fig:c_grundy} illustrates a connected Grundy coloring, which can be achieved using the connected sequence $\sigma_c$ to color each vertex respecting the first-fit heuristic property.

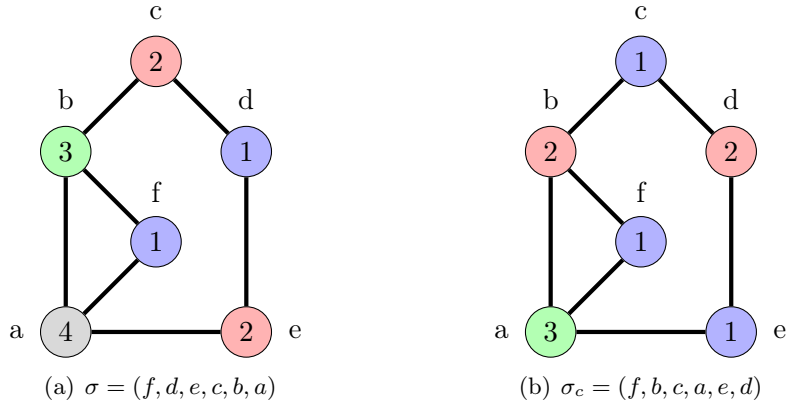
\begin{figure}[!ht]
    \centering
    \subfigure[$\sigma=(f,d,e,c,b,a)$]{
\begin{tikzpicture}
	[scale=1.2,every node/.style={circle,draw=black}]
	\node[label=left:a, fill=gray!30] (0) at  (0, 0)   {$4$};
	\node[label=b, fill=green!30] (1) at  (0, 2)   {$3$};
	\node[label=c, fill=red!30] (2) at  (1, 3)   {$2$};
	\node[label=d, fill=blue!30] (3) at  (2, 2)   {$1$};
        \node[label=right:e, fill=red!30] (4) at  (2, 0)   {$2$};
        \node[label=f, fill=blue!30] (5) at  (1, 1)   {$1$};
    \begin{scope}[line width=1.6pt,every node/.style={}]
	\draw (0) edge (1);
	\draw (1) edge (2);
	\draw (2) edge (3);
        \draw (3) edge (4);
        \draw (4) edge (0);
        \draw (0) edge (5);
        \draw (1) edge (5);
	\end{scope}
 \end{tikzpicture}
 \label{fig:n_grundy}
 }
 \hspace{1.5cm}
 \subfigure[$\sigma_c=(f,b,c,a,e,d)$]{
\begin{tikzpicture}
	[scale=1.2,every node/.style={circle,draw=black}]
	\node[label=left:a, fill=green!30] (0) at  (0, 0)   {$3$};
	\node[label=b, fill=red!30] (1) at  (0, 2)   {$2$};
	\node[label=c, fill=blue!30] (2) at  (1, 3)   {$1$};
	\node[label=d, fill=red!30] (3) at  (2, 2)   {$2$};
        \node[label=right:e, fill=blue!30] (4) at  (2, 0)   {$1$};
        \node[label=f, fill=blue!30] (5) at  (1, 1)   {$1$};
    \begin{scope}[line width=1.6pt,every node/.style={}]
	\draw (0) edge (1);
	\draw (1) edge (2);
	\draw (2) edge (3);
        \draw (3) edge (4);
        \draw (4) edge (0);
        \draw (0) edge (5);
        \draw (1) edge (5);
	\end{scope}
 \end{tikzpicture}
\label{fig:c_grundy}
 }
 \caption{Difference in the resulting colorings when using a non-connected and a connected sequence with the first-fit heuristic.}
    \label{fig:con_coloring}
\end{figure}

Combinatorial upper bounds have been proposed for $\Gamma(G)$. Notice that, as $\Gamma_c(G) \leq \Gamma(G)$, upper bounds for $\Gamma(G)$ are also valid for $\Gamma_c(G)$.
We refer the readers to the papers that follow for further details. 
\citet{Shietal05} defined the \textit{stair factor}, $\zeta(G)$. Given $u \in V(G)$, consider the set $N_{\leq}(u) = \{v \in V(G) : uv \in E(G),\ d_G(v) \leq d_G(u)\}$ and let $\Delta_2(G) = \max_{u\in V(G)} \max_{v\in N_{\leq}(u)} d(v)$. \citet{zaker2008new} demonstrated the validity of the bound $\Delta_2(G) + 1$. Recently, \citet{Silvaetal2024} defined a new upper bound, represented by $\Psi(G)$. The idea is to establish that a vertex can only receive a given color if the degrees of its neighbors form a sequence that allows the vertex to receive such a color. This can be calculated by defining a recursive function $\psi$, where $\psi(v,k)$ contains an upper bound for the largest color less than or equal to $k$ that the vertex $v$ can receive.

\subsection{Related works}

The concept of Grundy number was introduced by \citet{Gru39} in the context of game theory and digraphs \citep{Ber73}. The idea was formally introduced in graph theory by \citet{ChrSel79}, and independently proposed as the \textit{ochromatic number} by \citet{Sim83}. Later, \citet{ErdHarHedLas87} demonstrated that the Grundy and ochromatic numbers are equivalent. Several works focused on the complexity of the Grundy coloring problem for specific graph classes~\citep{hedetniemi1982linear,telle1997algorithms}. \citet{BonFouKimSik18} proved that the Grundy coloring problem can be solved exactly in $O(2.4443^{n})$ time using dynamic programming. Recently, \citet{Silvaetal2024} proposed integer programming (IP) formulations and a biased random-key genetic algorithm (BRKGA) to tackle the problem for general graphs. The authors also established the first benchmark set encompassing well-known instances in the graph coloring literature and new randomly generated graphs.

The connected Grundy coloring problem has been less explored in the literature. Connected Grundy colorings were studied in ~\citet{BenCamDouGriMorSamSil14}, where the authors showed that computing $\Gamma_c(G)$ is NP-hard even for chordal and complement of bipartite graphs. On the other hand, they proved that $\Gamma_c(G)=2$ for bipartite graphs. They also demonstrated that $\rchi_c(G) \leq \rchi(G)+1$ and that determining whether $\rchi_c(G)=\rchi(G)$ is NP-hard. \citet{MotRocSil20} focused on studying when $\rchi_c(G)=\rchi(G)$ and the decision problem $\rchi_c(G) \leq k$ for $H$-free graphs (a graph is $H$-free if it does not contain a copy of $H$ as an induced subgraph). \citet{BonGroMulNarPekWes21} studied connected edge colorings. \citet{BonFouKimSik18} proved that determining whether there is a connected Grundy coloring with $k = 7$ colors is already NP-complete. The authors also highlighted that the connected Grundy coloring problem is solvable in polynomial time when $k \leq 3$, but is still open for $4 \leq k \leq 6$.


\subsection{Main contributions and organization}

We now outline the main contributions of the paper.
We propose two IP formulations for the connected Grundy coloring problem.
The first one uses a standard way of partitioning the vertices into color classes while the second one employs the concept of representatives.
Moreover, partially motivated by the recent successful application of a BRKGA for the Grundy coloring problem~\citep{Silvaetal2024}, we elaborate a BRKGA that employs a new local search procedure. 
Besides, we perform a theoretical analysis of the neighborhood used in our local search.
To the best of our knowledge, our approaches are the first optimization methods for tackling the connected Grundy coloring problem for general graphs.


The remainder of the paper is organized as follows. 
Section~\ref{sec:formulations} describes the IP formulations.
Section~\ref{sec:brkga} details the proposed BRKGA.
Section~\ref{sec:localsearch} provides details of the employed local search procedure.
Section~\ref{sec:experiments} summarizes the performed computational experiments.
Section~\ref{sec:adaptinggrundy} shows how to adapt our newly proposed BRKGA to tackle the Grundy coloring problem.
Section~\ref{sec:conclusions} discusses some concluding remarks.
\section{Integer programming formulations}
\label{sec:formulations}


In this section, we present two IP formulations for the connected Grundy coloring problem. IP formulations were successfully used for solving several graph optimization problems~\citep{FurMalSan18, deFDiaMacSzw21,  MarMelRibSan22, MelRib23}. In what follows, Section~\ref{sec:formgrundycoloring} presents the standard formulation, while
Section~\ref{sec:repformulations} describes the formulation by representatives.

In the remainder of the paper, we will denote the set of vertices by $V=\{1,\ldots,n\}$. Furthermore, consider the sequence of available colors as $K= (1,\ldots,\min\{\zeta(G),\Psi(G)$, $\Delta_2(G)+1\})$, the sequence of possible colors for vertex $v$ by $K_v = (1,\ldots,\min\{\zeta(G),\psi(v,\Delta(G)+1),\Delta_2(G)+1\})$, and the sequence of possible colors of the vertex $v$ at time $t$ by $K_{vt} = \{k' \in K_v \ | \ k' \leq t\}$. We remind that the used bounds were mentioned in Section~\ref{sec:basicdefinitions}. 
Finally, we define $V_k = \{v \in V \ | \ k \in K_v \}$ as the set of vertices that can receive the $k$-th color. The proposed formulations use, for some variables, an index representing the time in which a vertex is colored. This time is associated with the position of the vertex in the determined connected sequence, i.e., a vertex colored at time $t$ is the $t$-th element in the sequence.

\subsection{Standard formulation}
\label{sec:formgrundycoloring}

To formulate the connected Grundy coloring problem as an integer program, consider the following decision variables: 

\[z_{vkt} = \left\{   
   \begin{array}{l}
      1, \textrm{ if vertex $v \in V$ receives color $k \in K_{vt}$ in time $t \in T$,}\\      
      0, \textrm{ otherwise;}
   \end{array}
   \right.
\]
\[ w_{k} = \left\{
   \begin{array}{l}
      1, \textrm{ if color $k \in K$ is used,}\\      
      0, \textrm{ otherwise.}
   \end{array}
   \right.
\]
Thus, the problem can be cast as:
\begingroup
\allowdisplaybreaks
\begin{flalign}
& \max  \sum_{k \in K} w_{k} & \label{cgru:obj}\\
& \sum_{\substack{ t \in T, \\ t \geq k}} z_{ukt} + \sum_{\substack{ t \in T, \\ t \geq k}} z_{vkt} \leq w_k,  \qquad \forall \ k \in K_v \cap K_u, \ uv \in E, & \label{cgru:0a} \\
& \sum_{t \in T} \sum_{k \in K_{vt}} z_{vkt} = 1,  \qquad \forall \ v \in V, & \label{cgru:0b} \\
& w_{k} \leq \sum_{v \in V} \sum_{\substack{ t \in T, \\ t \geq k}} z_{vkt} ,  \qquad \forall \ k \in K, & \label{cgru:0c} \\
&  \sum_{t' \in \{k',\ldots,t\}} z_{vk't'} \leq \sum_{\substack{ u \in N(v), \\ u \in V_k}} \sum_{t' \in \{k,\ldots,t-1\}} z_{ukt'},  \qquad \forall \ v \in V, \ k,k' \in K_v, t \in T\setminus\{1\}, \textrm{ with } k < k', \label{cgru:0d}\\
& \sum_{v \in V} \sum_{k \in K_{vt}} z_{vkt} = 1,  \qquad \forall \ t \in T, & \label{cgru:02} \\
& \sum_{k \in K_{vt}} z_{vkt} \leq \sum_{u \in N(v)} \sum_{t' =1}^{t-1} \sum_{k \in K_{ut'}} z_{ukt'},  \qquad \forall \ v \in V,\ t \in T\setminus\{1\}, & \label{cgru:bb} \\
& w_{k} \in \{0,1\},  \qquad \forall \ k \in K. & \label{cgru:09} \\ 
& z_{vkt} \in \{0,1\},  \qquad \forall \ v \in V, \ t \in T, \ k\in K_{vt}. & \label{cgru:07}
\end{flalign}
\endgroup

The objective function~\eqref{cgru:obj} maximizes the total number of colors used. Constraints~\eqref{cgru:0a} ensure that adjacent vertices do not receive the same color.
Constraints \eqref{cgru:0b} guarantee that each vertex receives exactly one color in a single period.
Constraints \eqref{cgru:0c} determine that $w_k$ is only set to one if color $k$ is used.
Constraints \eqref{cgru:0d} guarantee the Grundy property.
Notice that they imply that if a vertex $v\in V$ receives a color in periods one up to $t$, all the colors with lower index need to be used in the neighborhood of $v$ in periods one up to $t-1$. 
Constraints~\eqref{cgru:02} establish that a single vertex receives a color in each period.
Constraints~\eqref{cgru:bb} ensure that the coloring is connected.
Constraints~\eqref{cgru:09}-\eqref{cgru:07} define integrality requirements on the variables.
\subsection{Formulation by representatives}
\label{sec:repformulations}

Formulations by representatives~\citep{Camp04,FroMacNorRib2010} have been applied to successfully tackle graph coloring and other partitioning problems~\citep{Camp05,Bah14,MelRib15,MelQueSan21,MelRibRiv22}.
In what follows, we describe an IP formulation by representatives for the connected Grundy coloring problem.
Consider the following decision variables:
\[Z_{vut} = \left\{   
   \begin{array}{l}
      1, \textrm{ if vertex $u \in V$ is represented by vertex $v \in \bar{N}[u]$ in time $t \in T$, for $v \leq u$,}\\      
      0, \textrm{ otherwise;}
   \end{array}
   \right.
\]
\[ y_{vu} = \left\{
   \begin{array}{l}
      1, \textrm{ if vertices $v,u \in V$ are representatives and the color of $v$ precedes that of $u$, for $v\neq u$,}\\      
      0, \textrm{ otherwise.}
   \end{array}
   \right.
\]
An IP formulation by representatives for the problem can be defined as:
\begingroup
\allowdisplaybreaks
\begin{flalign}
& \max  \sum_{v \in V} \sum_{t\in T}Z_{vvt} &  \label{repcon:obj}\\
&  \sum_{t\in T}Z_{uvt} + \sum_{t\in T}Z_{uwt} \leq \sum_{t\in T}Z_{uut}, \qquad \forall \ u \in V, \ v,w \in \bar{N}[u],\ \textrm{ s.t. } vw \in E \textrm{ and } u \leq v < w, & \label{repcon:01} \\
&  \sum_{t\in T}Z_{uvt} \leq \sum_{t\in T}Z_{uut}, \qquad \forall \ u \in V,\ v \in \bar{N}(u), \textrm{ s.t. } N(v) \cap \bar{N}(u) = \emptyset \textrm{ and } u < v,  & \label{repcon:01b} \\
& \sum\limits_{ \substack{v \in \bar{N}[u], \\ v \leq u }}\sum_{t\in T} Z_{vut}  = 1, \qquad \forall \ u \in V, & \label{repcon:02} \\
&  \sum_{t'\in \{1,\ldots,t\}}Z_{uvt'} \leq \sum\limits_{\substack{w \in N(v) \cap \bar{N}[p], \\ p \leq w }} \sum_{t'\in \{1,\ldots,t-1\}}Z_{pwt'} + 1 - y_{pu}, \qquad \forall \ u,p \in V, \ v \in \bar{N}[u],\ & \nonumber \\
&  \begin{tabular}{@{}r@{}} $t \in T\setminus\{1\},\ \textrm{ s.t. }  p \neq u \textrm{ and } \ u \leq v,$\end{tabular} & \label{repcon:03} \\
& \sum_{\substack{u \in \bar{N}[v],\\ u\leq v}} Z_{uvt} \leq \sum_{v' \in N(v)}\sum_{\substack{u \in \bar{N}[v'],\\u\leq v'}} \sum_{t =1}^{t-1} Z_{uv't},  \qquad \forall \ v \in V,\ t \in T\setminus\{1\}, & \label{repcon:bb} \\
&  y_{vu} + y_{uv} \geq \sum_{t\in T}Z_{uut} + \sum_{t\in T}Z_{vvt} - 1, \qquad \forall \ u,v \in V, \ \textrm{ s.t. } u < v, & \label{repcon:04} \\
&  y_{uv} + y_{vu} \leq \sum_{t\in T}Z_{uut}, \qquad \forall \ u,v \in V, \ \textrm{ s.t. } u\neq v, & \label{repcon:05} \\
& \sum_{u \in V}\sum_{t\in T}Z_{uut} \leq \min(\Psi(G),\zeta(G),\Delta_2(G)+1), & \label{repcon:09} \\
&  Z_{uvt} \in \{0,1\}, \qquad \forall \ u \in V,\ v\in \bar{N}[u],\ t \in T,\ \textrm{ s.t. } u \leq v,  & \label{repcon:07} \\
&  y_{uv} \in \{0,1\}, \qquad \forall \ u,v \in V,\ \textrm{ s.t. }u\neq v. & 
\label{repcon:08}
\end{flalign}
\endgroup
The objective function \eqref{repcon:obj} maximizes the number of representative vertices.
Constraints \eqref{repcon:01} ensure that adjacent vertices do not receive the same color.
Constraints \eqref{repcon:01b} indicate that a vertex can only represent another one if the former is a representative.
Constraints \eqref{repcon:02} guarantee that every vertex receives a color in a single period.
Constraints \eqref{repcon:03} imply the Grundy property. They establish that if $p,u \in V$ are representatives, $p$ precedes $u$, and $u$ represents $v\in \bar{N}[u]$, then $p$ has to represent one of the neighbors of $v$ before $u$ can represent $v$.
Constraints~\eqref{repcon:bb} ensure that the coloring is connected.
Constraints \eqref{repcon:04}-\eqref{repcon:05} ensure an order between two vertices if and only if they are both representatives. Constraint \eqref{repcon:09} limits the total number of representative vertices.
Constraints \eqref{repcon:07}-\eqref{repcon:08} determine the integrality of the variables.

Figure~\ref{fig:rep_coloring_example} illustrates a coloring and its corresponding formulation by representatives' solution.

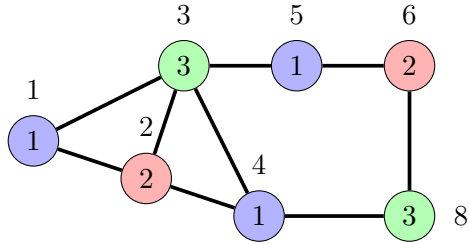
\begin{figure}[!ht]
    \centering
    \begin{tikzpicture}
	[scale=1.0,every node/.style={circle,draw=black}]
	\node[label=1, fill=blue!30, minimum size=0.6cm] (1) at  (0, 2)   {$1$};
	\node[label=2, fill=red!30, minimum size=0.6cm] (3) at  (1.5, 1.5)   {$2$};
        \node[label=3, fill=green!30, minimum size=0.6cm] (4) at  (2, 3)   {$3$};
        \node[label=4, fill=blue!30, minimum size=0.6cm] (5) at  (3, 1)   {$1$};
        \node[label=5, fill=blue!30, minimum size=0.6cm] (6) at  (3.5, 3)   {$1$};
        \node[label=6, fill=red!30, minimum size=0.6cm] (7) at  (5, 3)   {$2$};
        \node[label=right:8, fill=green!30, minimum size=0.6cm] (9) at  (5, 1)   {$3$};
    \begin{scope}[line width=1.4pt,every node/.style={}]
	\draw (1) edge (3);
        \draw (1) edge (4);
        \draw (3) edge (4);
        \draw (3) edge (5);
        \draw (4) edge (5);
        \draw (6) edge (7);
        \draw (6) edge (4);
        \draw (9) edge (7);
        \draw (9) edge (5);
	\end{scope}
 \end{tikzpicture}
    \caption{Coloring of a graph corresponding to the solution with the only non-zero values \mbox{$Z_{111} = Z_{222} = Z_{333} = Z_{144} = Z_{155}  = Z_{266} = Z_{277} = Z_{388} = 1$}, $Y_{12} = Y_{13} = Y_{23} = 1$. The external label represents the vertex index, while the internal label represents the color assigned to that vertex.}
    \label{fig:rep_coloring_example}
\end{figure}

\section{Biased random-key genetic algorithm}
\label{sec:brkga}


The BRKGA~\citep{gonccalves2011biased} is a general search metaheuristic designed to obtain high-quality solutions for challenging optimization problems. BRKGAs have been successfully applied to several problems~\citep{resende2012,GonWas20,johnson2020,HomFonGon23,SilRibSou23,MelRibRiv23}. We refer the reader to \citet{LonPesAndREs24a,LonPesAndRes24b} for recent surveys of the BRKGA literature.
Differently from standard genetic algorithms, BRKGAs make the representation and the intensification-diversification mechanism independent of the specific problem at hand.


Regarding the \textit{representation}, BRKGA solutions correspond to the characterization initially proposed for random-key genetic algorithms~\citep{bean1994genetic}. Namely, the solutions are represented as a vector of randomly generated real numbers in the interval $[0, 1)$, denoted by keys. Each random-key vector encodes a single solution to the problem. When it comes to the \textit{intensification},
in the mating process that generates the subsequent generation of solutions, one parent is always an elite solution, while the other is a non-elite. Besides, the elite parent has a higher probability of passing its characteristics (defined by its keys) to the offspring. The algorithm uses the parametric uniform crossover strategy~\citep{Spe91} to combine the two parent solutions when generating a new offspring.
With respect to the \textit{diversification}, a BRKGA introduces new random solutions, denoted by mutants, to the population in each generation. This is an attempt to prevent early convergence by enabling the algorithm to escape from local optima areas of the solution space.

As previously mentioned, BRKGA is a metaheuristic that can find high-quality solutions for the Grundy coloring problem \citep{Silvaetal2024}, but it tends to converge quickly. 
To address this, we enhance the BRKGA framework considered in the study of the Grundy coloring problem \citep{Silvaetal2024} by incorporating a reset mechanism and a local search procedure: the reset starts a new search, potentially exploring a different region of the solution space, while the local search utilizes available time to examine the neighborhood of a good solution in the current population, aiming to improve the best solution found.

Figure~\ref{fig:brkga-diagram} illustrates the flowchart of the local-search enhanced BRKGA framework used in our work. The initial population, denoted as $P$ (with $p$ initial solutions), consists entirely of randomly-generated random-key vectors. At each generation, the population is sorted according to the solutions' objective values. A fraction $p_e$ of the population is denominated ELITE, while the remaining ones are classified as NON-ELITE (see Figure \ref{fig:brkga_evolve}). The evolution phase between two generations works as follows: the elite population is entirely copied to the next generation, a percentage $p_m$ of the new population is formed by mutants, and the remaining portion is composed of new solutions generated using the \textit{intensification} strategy. If a new best solution is found in a generation, then a local search (see Section~\ref{sec:localsearch}) is performed on a number $b$ of solutions. After that, $b$ random-key vectors corresponding to the solutions obtained in the local search are created to replace the $b$ worst solutions of that generation. If $g_{lim}$ generations pass without improving the best-known solution, then the population is restarted keeping only the best-known solution for the next generation.

\begin{figure}[!ht]
    \centering
\includegraphics[height=11cm]{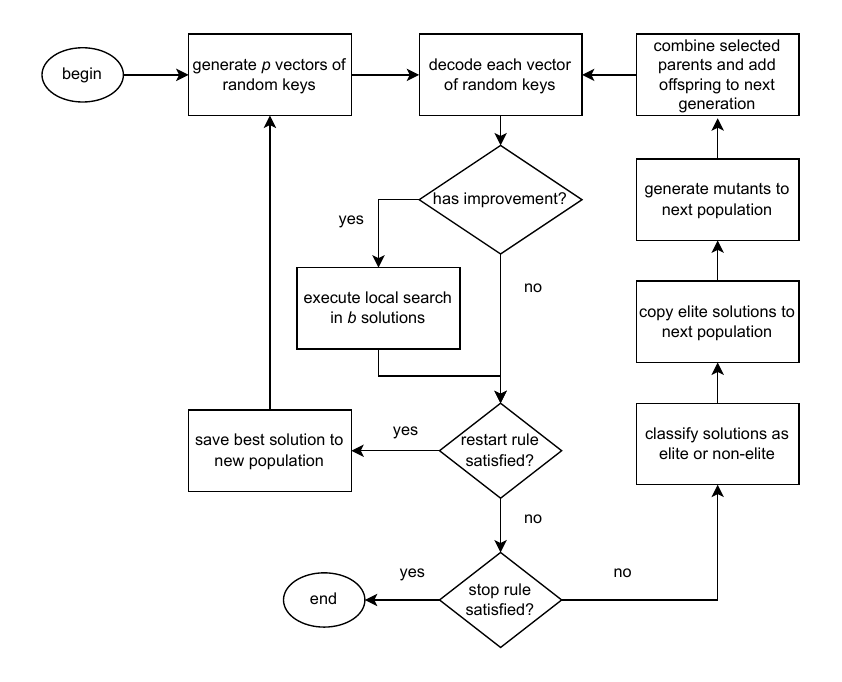}
    \caption{Framework of the BRKGA enhanced with local search.}
    \label{fig:brkga-diagram}
\end{figure}

\begin{figure}[!ht]
    \centering
    \includegraphics[height=8cm]{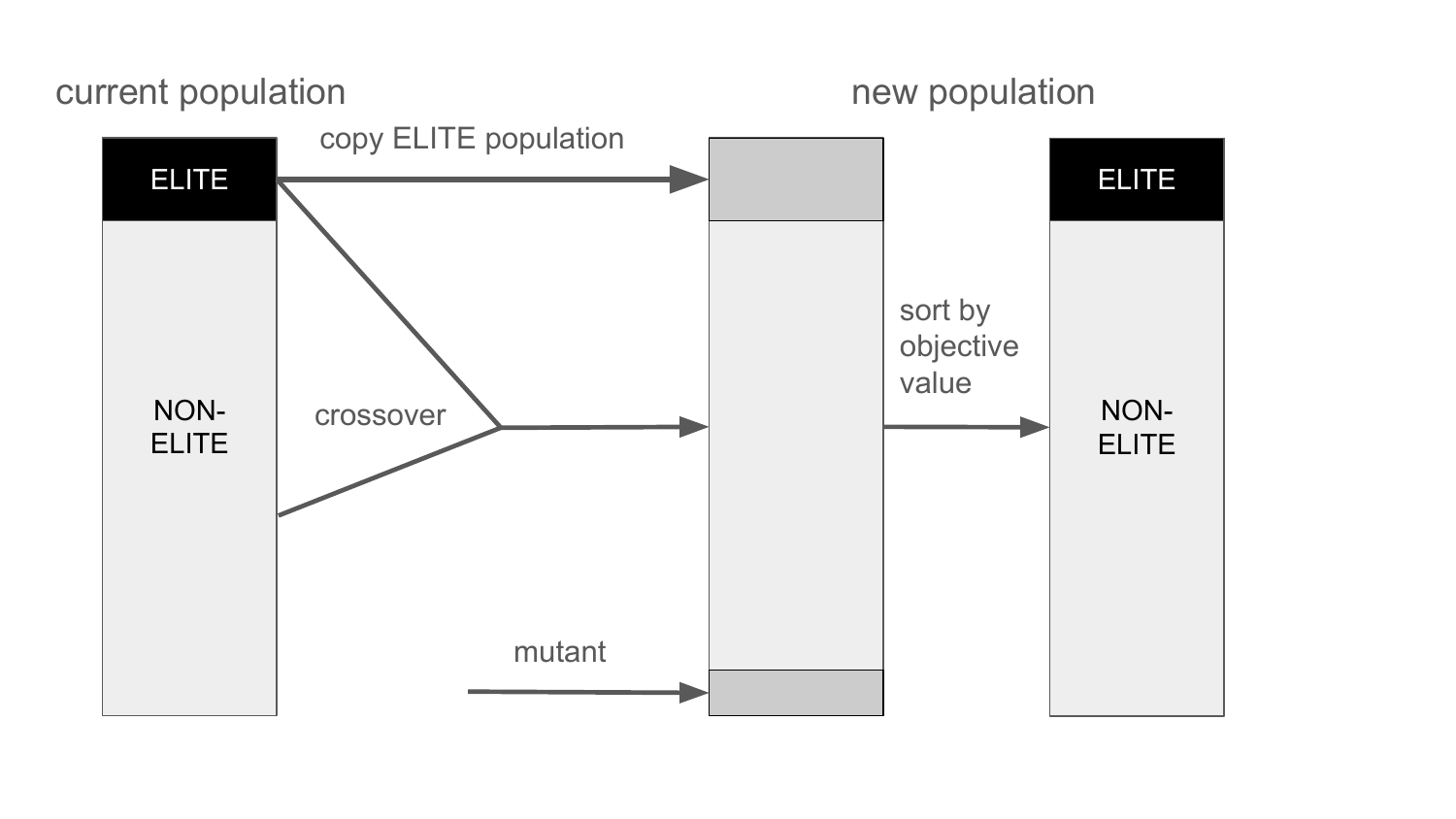}
    \caption{Population evolution between two generations of a BRKGA.}
    \label{fig:brkga_evolve}
\end{figure}

We refer the reader to \cite{gonccalves2011biased} for additional details of the framework. Notice that our BRKGA explanation is simplified to focus on the most relevant details and to highlight our contributions. 
A primary component of a BRKGA implementation is its \emph{decoder}, a deterministic algorithm responsible for mapping the random-key vector to a feasible solution for the problem at hand and computing its objective value.
In addition to the decoder, the BRKGA requires and is guided by the following parameters: the population size ($p$), the size of the elite population ($p_e$), the size of the mutant population ($p_m$), the elite inheritance probability ($\rho_e$), and the number of generations ($n_g$) or some other stopping criterion (such as maximum time), and the maximum number of generations without improvements before restarting ($g_{lim}$).


\subsection{Solution encoding}

Each solution is encoded as a random-key vector $x$ with length $\Upsilon = |V|$. In such a vector, the
$i$-th key represents the $i$-th vertex in $V$.

\subsection{Decoder}
\label{subsec:decoder}

The decoder defines the vertex coloring order considering the $i$-th key in the random-key vector representing a solution as the priority of the $i$-th vertex to be selected while respecting the connectivity constraint of the problem. 
Algorithm~\ref{alg:decoder_con_grundy} describes the decoder, which receives as inputs the graph $G$ and a random-key vector $x$. 
The algorithm utilizes a priority queue $Q$ based on the random keys $x_v$. The auxiliary procedure \mbox{Enqueue($Q$,$v$)} inserts the vertex $v$ into the priority queue $Q$ according to its key $x_v$. The Dequeue($Q$) method removes and returns the element with the highest key from $Q$.

In Algorithm~\ref{alg:decoder_con_grundy},
line~\ref{decc:init} initializes the color of all vertices as -1, indicating that they were not yet reached.
The vertex with the highest key is added to the priority queue $Q$ (lines~\ref{decc:maiorchave}-\ref{decc:enfileiramaiorchave}).
The loop of lines~\ref{decc:mainLoop}-\ref{decc:fimMainLoop} repeats the process of selecting the vertex with the highest key from the priority queue (line~\ref{decc:desenfileira}) and coloring it with the lowest possible color following the first-fit algorithm with the auxiliary procedure Color-Vertex (line~\ref{decc:colorefirstfit}). 
Color-Vertex sets the color of a vertex as the lowest-index color that has not yet been used for any of its neighbors.
Subsequently, its neighbors that have not yet been added to the priority queue are inserted into $Q$ (lines~\ref{decc:lacovizinhos}-\ref{decc:enfileravizinho}) and their colors are set to 0, representing that they were reached but are not yet colored.
Line \ref{decc:return} returns the vector of colors and the total number of used colors.


\begin{algorithm}[!ht]
\small
\caption {Decoder-Connected-Grundy ($G$, $x$) }
\label{alg:decoder_con_grundy}
    $colors \leftarrow \{-1,\ldots, -1\}$\; \label{decc:init}
    $v \leftarrow $ vertex with highest key value $x_v$\; \label{decc:maiorchave}
    Enqueue($Q$, $v$)\; \label{decc:enfileiramaiorchave}
    \While{$Q \neq \varnothing$} { \label{decc:mainLoop}
        $v \leftarrow$ Dequeue($Q$)\; \label{decc:desenfileira}
        $colors[v] \leftarrow$ Color-Vertex($G$, $v$, $colors$)\; \label{decc:colorefirstfit}
        \ForEach{$u \textit{ in } N(v)$}{\label{decc:lacovizinhos}
            \If{$colors[u] = -1$}{
                $colors[u] = 0$\;
                Enqueue($Q$, $v$)\; \label{decc:enfileravizinho}
            }
        }
    } \label{decc:fimMainLoop}
    \Return $colors$, $max_{v\in V}colors[v]$\; \label{decc:return}
\end{algorithm}

Figure~\ref{fig:example_decoder} illustrates the decoder's working mechanism. Figure~\ref{fig:decoder_randomkeys} depicts a random-key vector and the resulting connected sequence associated with the graph in Figure~\ref{fig:decoder_graph}. Figure~\ref{fig:decoder_graph} displays the resulting coloring. The first vertex selected is the one with the largest key ($f$), which receives color 1 by the first-fit heuristic. Next, the vertex $e$ is selected, which has the largest key among the neighbors of $f$. After that, vertex $c$ is the one with the largest key among the neighbors of $\{e,f\}$. Then, the vertex that has the largest associated key among the neighbors of those that have already been selected, $\{c,e,f\}$, is considered. Thus, vertex $a$ receives color 2. Finally, $b$ and $d$ are selected, in this order, based on their keys since the two are neighbors of at least one vertex that has already been selected.

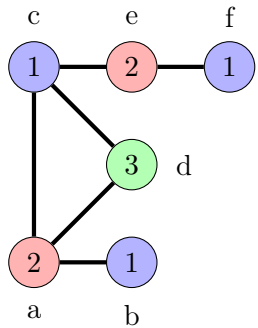
\begin{figure}[!ht]
   \centering
   \subfigure[Random-key vector and the corresponding connected sequence generated by the decoder.]{
    \includegraphics[height=4cm]{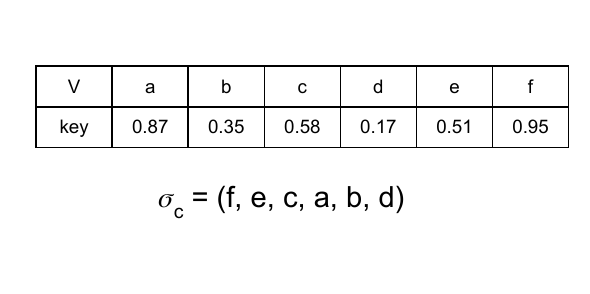}
   \label{fig:decoder_randomkeys}
   }
   \hspace{1cm}
    \subfigure[Colored input graph from the decoder execution on the random-key vector on the left.]{
\begin{tikzpicture}
	[scale=1.3,every node/.style={circle,draw=black}]
	\node[label=below:a, fill=red!30] (0) at  (0, 0)   {$2$};
	\node[label=below:b, fill=blue!30] (1) at  (1, 0)   {$1$};
	\node[label=c, fill=blue!30] (2) at  (0, 2)   {$1$};
	\node[label=right:d, fill=green!30] (3) at  (1, 1)   {$3$};
    \node[label=e, fill=red!30] (4) at (1,2)   {$2$};
    \node[label=f, fill=blue!30] (5) at (2,2)   {$1$};
    \begin{scope}[line width=1.6pt,every node/.style={}]
	\draw (0) edge (1);
	\draw (0) edge (2);
	\draw (0) edge (3);
        \draw (2) edge (3);
        \draw (2) edge (4);
        \draw (4) edge (5);
	\end{scope}
 \end{tikzpicture}
 \label{fig:decoder_graph}
 }
    \caption{Example of a random-key vector and the connected sequence generated by the decoder. On the right, the resulting coloring of the graph.}
    \label{fig:example_decoder}
\end{figure}

\begin{proposition}
    Algorithm~\ref{alg:decoder_con_grundy} can be implemented to run in $O(|V|\log |V| + |E|)$.
\end{proposition}
\begin{proof}
First, note that each vertex $v \in V$ enters and exits the priority queue $Q$ exactly once. Therefore, all operations related to $Q$ are performed in $O(|V|\log|V|)$. The while loop of lines~\ref{decc:mainLoop}-\ref{decc:fimMainLoop} involves traversing the graph's adjacency list to determine the color and enqueue the neighbors of each vertex. Thus, excluding the operations related to the priority queue, which have already been accounted for in the computational cost, the remaining operations can be performed in $O(|V| + |E|)$. Consequently, Algorithm~\ref{alg:decoder_con_grundy} can be implemented to run in $O(|V|\log |V| + |E|)$.
\end{proof}


\section{Local search}
\label{sec:localsearch}

In this section, we detail the local search procedure employed in our BRKGA.
We consider that a solution $S$ is represented by a connected vertex sequence $\sigma_c = (v_1,\ldots,v_n)$.
We say that a solution $S$ leads to, or produces, a connected Grundy coloring.
Let $c_S(v_i)$ be the color in the connected Grundy coloring led by solution $S$ corresponding to the vertex in it's $i$-th position, $1\leq i \leq n$. Notice that there may be multiple solutions that lead to the same coloring.

In what follows, Section~\ref{sec:neighborhood} introduces the used neighborhood.
Section~\ref{sec:neighborhoodanalysis} provides a theoretical analysis of the neighborhood.
Section~\ref{sec:NeighborhoodAlgorithm} describes the local search algorithm.
\subsection{Neighborhood}
\label{sec:neighborhood}

Given a solution $S$, define the \textit{move} operation as one that changes only the relative position of a single vertex $v_i$ with respect to the others to generate a new connected sequence. A neighbor $S'$ of $S$ is a solution that can be obtained from $S$ by applying a move operation.
Figure \ref{fig:neigh_example} illustrates some neighbors of a given solution. Note that not every move changes the resulting coloring, as exemplified in Figure \ref{fig:neigh2}. Furthermore, changes in the coloring may not lead to an increase in the total number of colors, as is the case of all the neighbors of the original solution in the example.  Consider the neighborhood $\mathcal{N}(S)$ of a solution $S$ as the set comprising all its neighbors. In other words,  $\mathcal{N}(S)$ is formed by all the solutions that can be obtained with a move operation over the sequence $S$. 

Note that we are only considering move operations that generate connected sequences. The neighborhood size can vary greatly depending on the graph structure. For a complete graph, it is possible to move any vertex to any position, in which case $|\mathcal{N}(S)| = O(V^2)$. On the other hand, for a path, if the first vertex in the sequence is one of the endpoints, then $|\mathcal{N}(S)| = 1$, if it is any other vertex $|\mathcal{N}(S)| = 2$. Section~\ref{sec:neighborhoodanalysis} discusses move operations that do not change the resulting coloring and the difficulty implied by the problem in which in some cases it may be necessary to completely recolor the graph.

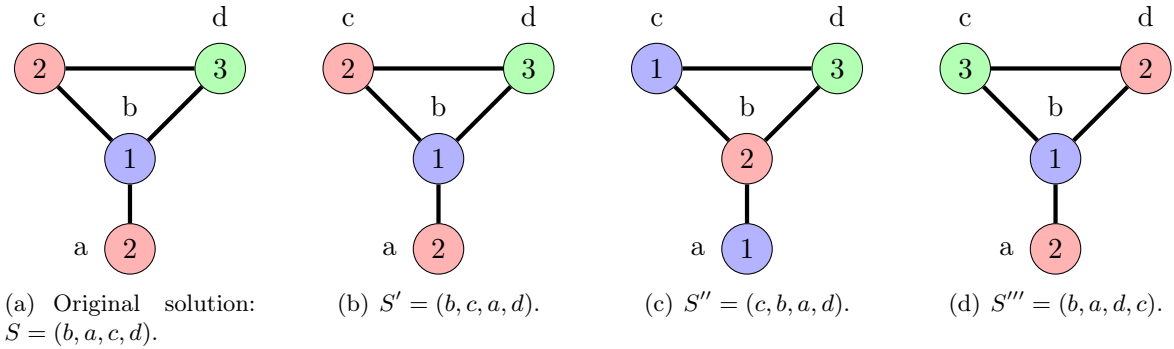
\begin{figure}[!ht]
    \centering
    \subfigure[Original solution: $S=(b,a,c,d)$.]{
\begin{tikzpicture}
	[scale=1.2,every node/.style={circle,draw=black}]
	\node[label=left:a, fill=red!30] (0) at  (0, 0)   {$2$};
	\node[label=b, fill=blue!30] (1) at  (0, 1)   {$1$};
	\node[label=c, fill=red!30] (2) at  (-1, 2)   {$2$};
	\node[label=d, fill=green!30] (3) at  (1, 2)   {$3$};
    \begin{scope}[line width=1.6pt,every node/.style={}]
	\draw (0) edge (1);
	\draw (1) edge (2);
	\draw (1) edge (3);
        \draw (2) edge (3);
	\end{scope}
 \end{tikzpicture}
 \label{fig:neigh1}
 }
 \hspace{0.5cm}
 \subfigure[$S'=(b,c,a,d)$.]{
\begin{tikzpicture}
	[scale=1.2,every node/.style={circle,draw=black}]
	\node[label=left:a, fill=red!30] (0) at  (0, 0)   {$2$};
	\node[label=b, fill=blue!30] (1) at  (0, 1)   {$1$};
	\node[label=c, fill=red!30] (2) at  (-1, 2)   {$2$};
	\node[label=d, fill=green!30] (3) at  (1, 2)   {$3$};
    \begin{scope}[line width=1.6pt,every node/.style={}]
	\draw (0) edge (1);
	\draw (1) edge (2);
	\draw (1) edge (3);
        \draw (2) edge (3);
	\end{scope}
 \end{tikzpicture}
 \label{fig:neigh2}
 }
 \hspace{0.5cm}
 \subfigure[$S''=(c,b,a,d)$.]{
\begin{tikzpicture}
	[scale=1.2,every node/.style={circle,draw=black}]
	\node[label=left:a, fill=blue!30] (0) at  (0, 0)   {$1$};
	\node[label=b, fill=red!30] (1) at  (0, 1)   {$2$};
	\node[label=c, fill=blue!30] (2) at  (-1, 2)   {$1$};
	\node[label=d, fill=green!30] (3) at  (1, 2)   {$3$};
    \begin{scope}[line width=1.6pt,every node/.style={}]
	\draw (0) edge (1);
	\draw (1) edge (2);
	\draw (1) edge (3);
        \draw (2) edge (3);
	\end{scope}
 \end{tikzpicture}
\label{fig:neigh3}
 }
 \hspace{0.5cm}
 \subfigure[$S'''=(b,a,d,c)$.]{
\begin{tikzpicture}
	[scale=1.2,every node/.style={circle,draw=black}]
	\node[label=left:a, fill=red!30] (0) at  (0, 0)   {$2$};
	\node[label=b, fill=blue!30] (1) at  (0, 1)   {$1$};
	\node[label=c, fill=green!30] (2) at  (-1, 2)   {$3$};
	\node[label=d, fill=red!30] (3) at  (1, 2)   {$2$};
    \begin{scope}[line width=1.6pt,every node/.style={}]
	\draw (0) edge (1);
	\draw (1) edge (2);
	\draw (1) edge (3);
        \draw (2) edge (3);
	\end{scope}
 \end{tikzpicture}
\label{fig:neigh4}
 }
 \caption{Example of the solutions obtained by applying a move operation on vertex $c$ in the original solution $S$.}
    \label{fig:neigh_example}
\end{figure}
\subsection{Neighborhood analysis}
\label{sec:neighborhoodanalysis}

In this section, we examine the neighborhood defined in Section~\ref{sec:neighborhood}.  We identify cases in which moving a vertex does not create a new coloring and cases where it may generate a coloring with more colors. Additionally, we explore properties arising from the problem characteristics to reduce the neighborhood size and make the search more efficient.

Define a function $p_{S}: V \rightarrow \{1,\ldots,|V|\}$ that maps each vertex to its position in the solution $S$. Define a function $f_{S}: V \rightarrow \{1,\ldots,|V|\}$ that maps each vertex to the position of its first neighbor in the solution $S$ and a function $fc_{S}: V \rightarrow \{1,\ldots,|V|\}$ that maps, for each vertex $u$ in the solution $S$, the total number of neighbors that appear before $u$. We assume that a query in this function can be done in $O(1)$. More information on how to keep a structure representing such a function will be provided in Section~\ref{sec:NeighborhoodAlgorithm}. Given two solutions $S$ and $S'$, we say that a subset $V' \subseteq V$ has its relative order unchanged if and only if $\forall u,v \in V'$ if $p_{S}(v) < p_{S}(u)$ implies that $p_{S'}(v) < p_{S'}(u)$. In such a case, we also say that the sequences $S$ and $S'$ are {\em stable} with respect to $V'$, or that they {\em agree} with respect to $V'$.


\begin{lemma}
\label{lemma:samecolor}
    Consider a solution $S$, a vertex $v$, and a solution $S'$ obtained by the move operation applied to the vertex $v$. If $p_{S}(v) < p_{S'}(v)$, every vertex in the interval $[1,p_{S}(v))$ in the coloring produced by $S'$ will have the same color as that in the coloring produced by $S$. If $p_{S}(v) > p_{S'}(v)$, every vertex in the interval $[1,p_{S'}(v))$ in the coloring produced by $S'$ will have the same color as that produced by $S$.
\end{lemma}

\begin{proof}
    The proof is direct by considering how the vertices are colored. For the first claim, the subsequence of vertices is the same in $S$ and $S'$ in the interval $[1, p_S(v))$. Consequently, the colors of the vertices in this interval will be the same, since the first-fit is a deterministic heuristic. The same applies to the second claim, i.e., the sequence is preserved until the position in which the move operation changes the sequence.
\end{proof}

\begin{corollary}
\label{col:minequalcolors}
    Consider a solution $S$, a vertex $v$, and a solution $S'$ obtained by the move operation applied to the vertex $v$. It follows that the first $min(p_{S'}(v),p_{S}(v)))-1$ vertices will not have their colors changed.
\end{corollary}

\begin{corollary}
    \label{col:maxdiffcolors}
    Consider a solution $S$, a vertex $v$, and a solution $S'$ obtained by the move operation applied to the vertex $v$. It follows that the number of vertices that can have their color changed is $O(|V|-(min(p_{S'}(v),p_{S}(v))-1))$.
\end{corollary}


\begin{proposition}
\label{lemma:worstcase}
    In the worst case, applying a move operation to a vertex implies updating the color of all vertices.
\end{proposition}

\begin{proof}
    Figure \ref{fig:changeAllColors} illustrates a case in which changing a vertex position in the sequence can lead to the update of all vertices. 
    In the example, moving the first vertex ($a$) to the end of the sequence implies that all the vertices must be recolored.
\end{proof}

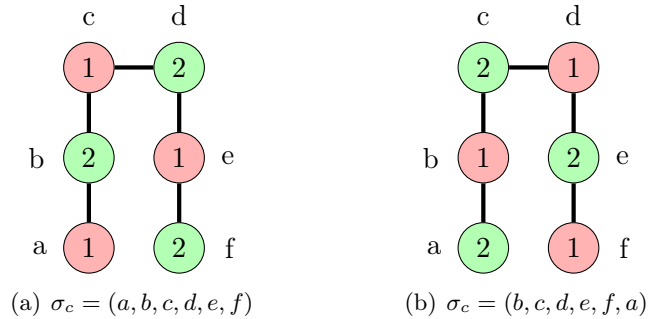
\begin{figure}[!ht]
    \centering
    \subfigure[$\sigma_c=(a,b,c,d,e,f)$]{
\begin{tikzpicture}
	[scale=1.2,every node/.style={circle,draw=black}]
	\node[label=left:a, fill=red!30] (0) at  (0, 0)   {$1$};
	\node[label=left:b, fill=green!30] (1) at  (0, 1)   {$2$};
	\node[label=c, fill=red!30] (2) at  (0, 2)   {$1$};
	\node[label=d, fill=green!30] (3) at  (1, 2)   {$2$};
        \node[label=right:e, fill=red!30] (4) at  (1, 1)   {$1$};
        \node[label=right:f, fill=green!30] (5) at  (1, 0)   {$2$};
    \begin{scope}[line width=1.6pt,every node/.style={}]
	\draw (0) edge (1);
	\draw (1) edge (2);
	\draw (2) edge (3);
        \draw (3) edge (4);
        \draw (4) edge (5);
	\end{scope}
 \end{tikzpicture}
 \label{fig:n_cgrundy}
 }
 \hspace{1.5cm}
 \subfigure[$\sigma_c=(b,c,d,e,f,a)$]{
\begin{tikzpicture}
	[scale=1.2,every node/.style={circle,draw=black}]
	\node[label=left:a, fill=green!30] (0) at  (0, 0)   {$2$};
	\node[label=left:b, fill=red!30] (1) at  (0, 1)   {$1$};
	\node[label=c, fill=green!30] (2) at  (0, 2)   {$2$};
	\node[label=d, fill=red!30] (3) at  (1, 2)   {$1$};
        \node[label=right:e, fill=green!30] (4) at  (1, 1)   {$2$};
        \node[label=right:f, fill=red!30] (5) at  (1, 0)   {$1$};
    \begin{scope}[line width=1.6pt,every node/.style={}]
	\draw (0) edge (1);
	\draw (1) edge (2);
	\draw (2) edge (3);
        \draw (3) edge (4);
        \draw (4) edge (5);
	\end{scope}
 \end{tikzpicture}
\label{fig:c_cgrundy}
 }
 \caption{Example of how moving a vertex can lead to changing the color of all vertices.}
\label{fig:changeAllColors}
\end{figure}

\begin{proposition}
    \label{lemma:intervalEqualLeft}
    Consider a solution $S$, a vertex $v$, and a solution $S'$ obtained by the move operation applied to the vertex $v$. If $p_{S'}(v) < p_{S}(v)$ and $c_{S}(v) = c_{S'}(v)$, all the vertices will have the same color in $S$ and $S'$.
\end{proposition}

\begin{proof}
Consider two solutions $S$ and $S'$ for which the conditions hold. Using Lemma~\ref{lemma:samecolor}, we can state that the vertices in the interval $[1,p_{S'}(v)]$ will have the same colors in $S$ and $S'$. 
Assume, by contradiction, that there is a vertex $w$ such that $p_{S'}(w) = p_{S'}(v)+1$ and $c_{S}(w) \neq c_{S'}(w)$. This implies that the color $c_{S}(w)$ has to be used in the neighborhood of $w$ before $w$ is colored in $S'$. However, we know that we have the same vertices and they are colored with the same colors as in $S$ in the interval $[1,p_{S'}(v)]$. Therefore, this leads to a contradiction because $w$ and one of its neighbors in $S$ would have to have the same color. This process can be repeated for other vertices, proving that, given the conditions of the lemma, all the vertices in the solution $S'$ will have the same color as in $S$.
\end{proof}

\begin{proposition}
    \label{lemma:intervalEqualRight}
    Consider a solution $S$, a vertex $v$, and a solution $S'$ obtained by the move operation applied to the vertex $v$. If $p_{S}(v) < p_{S'}(v)$ and $c_{S}(u) = c_{S'}(u)$ for every $u\in N[v]$ such that $p_{S}(v) \leq p_{S'}(u) \leq p_{S'}(v)$, then all vertices will have the same color in $S$ and $S'$.
\end{proposition}

\begin{proof} 
Consider two solutions $S$ and $S'$ for which the conditions hold. Using Lemma~\ref{lemma:samecolor}, it follows that the vertices in the interval $[1,p_{S} (v))$ will have the same colors in $S$ and $S'$. 
Moreover, the vertices $u\notin N[v]$ such that $p_{S}(v) \leq p_{S'}(u) \leq p_{S'}(v)$ also have the same colors in $S$ and $S'$ since none of their neighbors anteceding them in the sequence changed color. Thus, every vertex in the interval $[p_{S} (v),p_{S'}(v)]$ have the same colors in $S$ and $S'$. It remains to show that the vertices in the interval $(p_{S'}(v),|V|]$ have the same colors in both solutions.  The result follows since there was no change in the colors of their neighbors preceding them in the sequence.
\end{proof}

\begin{proposition}
\label{lemma:criticalposition}
Consider a solution $S$, a vertex $v$, and two vertices $z,w\in N(v)$ such that $p_{S}(z) < p_{S}(w)$ and there is no $u \in N(v)$ with $p_S(u) \in (p_{S}(z),p_{S}(w))$. It follows that all the solutions $S'$ that are stable with respect to $V - v$ and in which $p_{S'}(z) < p_{S'}(v) < p_{S'}(w) $ result into the same connected Grundy coloring.
\end{proposition}
\begin{proof}
By contradiction, assume that there are vertices $v,w,z$ and two solutions $S'$ and $S''$ respecting the conditions of the lemma such that $S'$ and $S''$ lead to different connect Grundy colorings. Let $u$ be the first vertex in the related ordering for which $c_{S'}(u) \not= c_{S''}(u)$. Notice that $u$ must be a neighbor of $v$ or $v$ itself, since $S'$ and $S''$ are stable with respect to $V - v$. Observe that all the vertices in the neighborhood of $v$ that precede  $w$ in $S'$ and $S''$ have the same set of preceding vertices and in the same order. Thus, they must all receive the same color in $c_{S'}$ and $c_{S''}$. Besides, so does $w$ and for extension $v$, as $v$ does not have any other neighbor before it and all such neighbors have the same color. Finally, by the same argument, $z$ must have the same color and all other neighbors of $v$ succeeding it.
\end{proof}

Proposition~\ref{lemma:criticalposition} implies the following corollary stating a necessary condition for neighbor solutions to have different colors.

\begin{corollary}
\label{col:criticalpositions}
Consider a solution $S$ and a neighbor solution $S' \in \mathcal{N}(S) $. It follows that $S$ and $S'$ can only lead to different colorings if they differ in the relative order of at least a pair of neighbor vertices.
\end{corollary}

\begin{proposition}
    \label{lemma:connectedLeft}
    Consider a connected solution $S$ and two vertices $u,v$ such that $uv \in E$ and $p_{S}(u) < p_{S}(v)$. The move of the vertex $v$ to the position $p_{S}(u)$ always generates a connected solution $S'$ if $f_{S}(v) < p_{S}(u)$ or $p_S(u)=1$. Furthermore, it is possible to verify that this operation generates a connected solution $S'$ in $O(1)$.
\end{proposition}

\begin{proof}
    First, assume that $p_S(u)>1$ and $f_S(v) < p_S(u)$. Since $S$ is connected, then the interval $[1,p_S(u))$ in $S$ is connected and it is the same as in $S'$. Since $f_S(v) < p_S(u)$, then $v$ has a neighbor in the interval $[1,p_S(u))$. As $p_S(u) = p_{S'}(v )$, we have that $S'$ is connected in the interval $[1,p_{S'}(v)]$. Moreover, as $S$ is connected, the interval $(p_S(u),p_S(v))$ is connected in $S'$, as well as the interval $[1, p_S(v)]$. Finally, knowing that the set of vertices in the interval $[1,p_S(v)]$ is equal in $S$ and $S'$, and $S$ is connected, then $S'$ is connected in the interval $ [1,|V|]$, that is, $S'$ is connected. As already seen, the check that needs to be done is whether $f_S(v) < p_S(u)$, which is a comparison that is made in $O(1)$. 
    We now consider the case $p_S(u)=1$, which implies that $p_{S'}(v)=1$. Since $S$ is connected, $S'$ is connected in the interval $[1,p_S(v)]$. Besides, as $uv \in E$, the sequence in the interval $[1,2]$ is connected. Additionally, the relative order of the remaining vertices belonging to $V-v$ in the two solutions is the same in the interval $[3,|V|]$. Thus, $S'$ will always be connected and no further check needs to be done.
\end{proof}

\begin{proposition}
    \label{lemma:connectedRight}
    Consider a connected solution $S$ and two vertices $u,v$ with $uv \in E$ such that $p_{S}(v) < p_{S}(u)$. The move of the vertex $v$ to the position $p_{S}(u)$ always generates a connected solution $S'$ if, for every $w \in N(v)$ such that $p_{S}(v)+1 \leq p_{S}(w) \leq p_{S}(u)$, one of the following conditions holds: (i) $f_{S}(w) < p_{S}(v)$ or (ii) $f_{S}(w) = p_{S}(v)$ and $fc_{S}(w) > 1$. Moreover, we can verify whether this operation generates a connected solution $S'$ in $O(|N(v)|)$.
\end{proposition}

\begin{proof}
    Consider a connected solution $S$ and two vertices $u,v$. First, consider the case $p_{S}(v)>1$. We want to show that applying the move operation in the vertex $v$ to the position $p_{S}(u)$ will generate a connected sequence $S'$ as long as one of the conditions of the lemma hold. By hypothesis, as $S$ is connected, we have that the sequence $S'$ is connected in the interval $[1, \ldots, p_{S}(v))$. For the interval $(p_{S}(v) ,p_{S}(u)]$, we need to check for every $ w \in N(v)$, if they already have a neighbor that appears before them for the sequence $S'$ to be connected. Note that for the sequence $S'$ we are now showing that the sequence $(v_1,\ldots,u,v)$ is connected. Considering the first case, if $f_{S}(w) < p_{S}(v)$, then $w$ has a neighbor that appears before $v$ so it will still be connected to at least that neighbor. Now, in the second case, $f_{S}(w) = p_{S}(v)$ which means that $v$ is the first neighbor of $w$ in the sequence and to stay connected it needs another of its neighbors to appear before $w$ and after $v$, that is, $fc_{S}(w) > 1$, so $w$ does not depend exclusively on $v$ to be connected and its movement will not cause it to become disconnected in the induced subgraph. Thus, the interval $[1,p_{S}(u)]$ is connected in $S'$. Finally, as $S$ is connected, the interval $[1,p_{S}(u)]$ contains the same vertices in $S$ and in $S'$, and the order of the remaining vertices in the two solutions is the same, so in this case $S'$ will be connected. 
    We now consider the case  $p_S(v)=1$. The vertex $w$ such that $p_S(w)=2$ is going to be the first in the sequence and it is a trivial connected sequence. For the remaining vertices, the same reasoning of the previous case applies. 
    In total we only need to check the neighbors of $v$ and each check can be done in constant time. Thus, the whole procedure can be performed in $O(|N(v)|)$.

\end{proof}

\subsection{Local search algorithm}
\label{sec:NeighborhoodAlgorithm}

This section describes a local search algorithm considering the neighborhood defined in Section~\ref{sec:neighborhood}. It employs the first-improvement strategy, and uses the decoder described in Section~\ref{subsec:decoder} as well as Algorithm~\ref{alg:colorsequence} (detailed in the following) to reconstruct the solutions in the neighborhood. 
The idea is to move a single vertex to force it to come before or after one of its neighbors and thus change the colors of the other vertices. 
We already argued that it can be computationally expensive to use an approach that has to recolor a solution entirely or partially to know whether there will be an improvement or not. Moreover, from Proposition~\ref{lemma:worstcase} we know that generating a new coloring can be computationally expensive.
To minimize the number of solutions that will be recolored, only changes that could potentially result in a solution with more colors will be considered (supported by Proposition~\ref{lemma:criticalposition} and Corollary~\ref{col:criticalpositions}).

Algorithm \ref{alg:colorsequence} performs the move operation of a vertex to the given position, resulting in a new coloring. It receives as inputs the graph $G$, a solution $S$, a vertex $v$ to be moved, the new position to which the vertex will be moved ($nPosition$), and the last position whose color is guaranteed not to change due to the movement ($fPosition$). 
Line \ref{cs:move} generates the new solution by moving the vertex, while line \ref{cs:copy} copies the colors from the previous solution for all the vertices before the critical interval. Line \ref{cs:initCheck} creates a variable to track whether any vertex's color has changed within the critical interval. The loop in lines \ref{cs:initLoop}-\ref{cs:endLoop} iterates through all the remaining vertices in the sequence to assign colors. Lines \ref{cs:initNotChange}-\ref{cs:endNotChange} are executed while all vertices are the same colors as in the original sequence. Line \ref{cs:checkInterval} checks if the algorithm has exited the critical interval without any vertex having its color changed or the current vertex is not a neighbor of the moved vertex. If so, the vertex will retain the same colors as in the previous solution, which is implemented in line \ref{cs:copyOutCritical}. Lines \ref{cs:initcheckN}-\ref{cs:endcheckN} are executed whenever the current vertex is adjacent to or is the vertex passed as a parameter within the critical interval. In this case, it assigns a new color to the vertex and checks if the color differs from the previous solution. If it does, all subsequent vertices must be recolored (as it is within the critical interval). Lines \ref{cs:alrChange} and \ref{cs:alrChangeColor} color the current vertex whenever a vertex in the critical interval has already been colored differently from the previous solution. Finally, the algorithm returns in line \ref{cs:return} the new solution, the total number of colors used, and the coloring.

\begin{algorithm}
\small
\caption {ColorSequence ($G$, $S$, $colors$, $v$, $nPosition$, $fPosition$) }
\label{alg:colorsequence}
$S' \leftarrow move(S, v, nPosition)$\;\label{cs:move}
$colors' \leftarrow copyColors(colors, fPosition)$\;\label{cs:copy}
$checkChange \leftarrow FALSE$\;\label{cs:initCheck}
\ForEach{$p \in (fPosition,|V|]$}{\label{cs:initLoop}
    \If{$checkChange = FALSE$}{\label{cs:initNotChange}
        \If{$p > nPosition$ or $S'[p] \notin N[v]$}{\label{cs:checkInterval}
            $colors'[S'[p]] \leftarrow colors[S'[p]]$\;\label{cs:copyOutCritical}
        }
        \Else{\label{cs:initcheckN}
        $colors'[S'[p]] \leftarrow$ Color-Vertex$(G,S'[p], colors')$\;\label{cs:tryColor}
        \If{$colors'[S'[p]] \neq colors[S'[p]]$}{
            $checkChange \leftarrow TRUE$\;
        }
        }\label{cs:endcheckN}
    }\label{cs:endNotChange}
    \Else{\label{cs:alrChange}
        $colors'[S'[p]] \leftarrow$ Color-Vertex$(G,S'[p], colors')$\;\label{cs:alrChangeColor}
    }
}\label{cs:endLoop}

\Return $S', \max_{v \in V}colors', colors'$\;\label{cs:return}
\end{algorithm}

Algorithm~\ref{alg:decoder_con_grundy_1opt} details the local search procedure. It receives as inputs the graph $G$ and a random-key vector $x$.  
Line \ref{deccopt:initialSolution} initializes the solution from $x$ using the BRKGA's decoder. Line \ref{deccopt:seq} generates the sequence encoded by $x$. Line \ref{deccopt:checkControl} initializes a control variable defining whether there is an improvement in the local search process.
The local search is performed in the loop of lines \ref{deccopt:iniWhileLoop}-\ref{deccopt:endWhileLoop}. Line \ref{deccopt:resethasincrease} resets the control variable. Line~\ref{deccopt:auxStructs} creates the auxiliary structures $f_{S}$, $fc_{S}$ and $p_{S}$.
Remind that the vector $f_{S}$ stores at position $v$ the position of the first neighbor of $v$ that appears in the solution $S$. The vector $p_{S}$ stores the position of the vertex $v$ in the solution $S$. Moreover, the vector $fc_{S}$ stores, at position $v$, the number of neighbors of $v$ that appears before it in the solution $S$. Note that these structures play the role of the functions $p_{S}$, $f_{S}$, and $fc_{S}$  defined in Section~\ref{sec:neighborhoodanalysis}. Besides, $p_{S}$ can be computed in $O(|V|)$ through an iteration over the sequence $\sigma_c$, and the others in $O(|V|+|E|)$ using a graph traversal based in the sequence.
The loop in lines \ref{deccopt:iniForLoop}-\ref{deccopt:endForLoop} iterates for each vertex to try to change its position in the sequence. The inner loop in lines \ref{deccopt:iniInForLoop}-\ref{deccopt:endInForLoop} iterates for each neighbor of the vertex that will be a possible reference for the new position.
Note that a first improvement approach is being used in these loops. Lines \ref{deccopt:checkCriticalPositionBefore} and  \ref{deccopt:checkCriticalPositionAfter} check whether it is possible to perform the \textit{move operation} to the specific critical position. The functions \textit{canMoveLeft} and \textit{canMoveRight} perform the connectivity check on what would be the resulting sequence, based on Propositions \ref{lemma:connectedLeft} and \ref{lemma:connectedRight}, respectively. In lines \ref{deccopt:reconstructSolBefore} and \ref{deccopt:reconstructSolAfter}, Algorithm \ref{alg:colorsequence} is executed to generate the new solution $S'$ and its corresponding coloring based on the updated sequence. Lines~\ref{deccopt:checkBefore}-\ref{deccopt:updateBefore} and lines~\ref{deccopt:checkAfter}-\ref{deccopt:updateAfter} check whether there has been an improvement. In an affirmative case, they update the current solution. Line \ref{deccopt:return} returns the best solution found during the local search.

\begin{algorithm}[!ht]
\small
\caption {LocalSearch ($G$, $x$) }
\label{alg:decoder_con_grundy_1opt}
    $colors, max \leftarrow $ Decoder-Connected-Grundy($G$, $x$)\;
    \label{deccopt:initialSolution}
    $S \leftarrow sequence(x)$\; \label{deccopt:seq}
    $hasIncrease \leftarrow \mathit{TRUE}$ \;  \label{deccopt:checkControl}
    \While{$hasIncrease = \mathit{TRUE}$}{  \label{deccopt:iniWhileLoop}
        $hasIncrease \leftarrow \mathit{FALSE}$\; \label{deccopt:resethasincrease}
        $f_{S}, fc_{S}, p_{S} \leftarrow computeFirstNeighborAndVerticesPosition(S)$\; \label{deccopt:auxStructs}
        \ForEach{$v \in V$ \textbf{while} $hasIncrease = \mathit{FALSE}$}{ \label{deccopt:iniForLoop}
            \ForEach{$u \in N(v)$ \textbf{while} $hasIncrease = \mathit{FALSE}$}{ \label{deccopt:iniInForLoop}
                \If{$p_{S}[u] < p_{S}[v]$ and $canMoveLeft(G, S, f_{S}, fc_{S}, p_{S}, u, v)$}{ \label{deccopt:checkCriticalPositionBefore}
                 $S', colors', max' \leftarrow $ ColorSequence($G$, $S$, $colors$, $v$, $p_{S}[u]$, $p_{S}[u]-1$)\; \label{deccopt:reconstructSolBefore}
                 \If{$max < max'$}{\label{deccopt:checkBefore}
                    $S, max, colors \leftarrow S', max', colors'$\;
                    $hasIncrease \leftarrow \mathit{TRUE}$\;
                 }\label{deccopt:updateBefore}
                } 

                \If{$p_{S}[u] > p_{S}[v]$ and $canMoveRight(G, S, f_{S}, fc_{S}, p_{S}, u, v)$}{\label{deccopt:checkCriticalPositionAfter}
                 $S', colors', max' \leftarrow $ ColorSequence($G$, $S$, $colors$, $v$, $p_{S}[u]$, $p_{S}[v]-1$)\; \label{deccopt:reconstructSolAfter}
                 \If{$max < max'$}{\label{deccopt:checkAfter}
                    $S, max, colors \leftarrow S', max', colors'$\;
                    $hasIncrease \leftarrow \mathit{TRUE}$\;
                 }\label{deccopt:updateAfter}
                } 
                
            }\label{deccopt:endInForLoop}
        }\label{deccopt:endForLoop}
    }\label{deccopt:endWhileLoop}
    \Return $S$\; \label{deccopt:return}
\end{algorithm}
\section{Computational experiments}
\label{sec:experiments}

All the experiments were executed on a machine running Ubuntu x86-64 GNU/Linux, with an Intel Core i7-10700 Octa-Core 2.90 Ghz processor and 16Gb of RAM. The formulations were implemented in Julia and solved with Gurobi 10.0.1. The BRKGA was developed in C++ using the BRKGA API \citep{apiBRKGApaper,apiBRKGA}. 

\subsection{Benchmark instances}
\label{sec:instances}
The reference dataset includes graphs previously utilized in existing literature for other coloring concerns~\citep{MelQueSan21, Silvaetal2024}. It encompasses (a) random graphs, (b) geometric graphs, (c) bipartite graphs, (d) the complements of bipartite graphs, and (e) examples sourced from the second \textit{DIMACS Implementation Challenge}~\citep{dimacs}.

Instances (a)-(c) were produced using the graph generator \textit{ggen}~\citep{ggen} by \citet{MelQueSan21}. This set has been extended by \citet{Silvaetal2024} to include smaller graphs with up to 30 vertices. They encompass $|V| \in \{15, 20, 25, 30, 50, 60, 70, 80\}$ and were constructed employing $\eta \in \{0.2, 0.4, 0.6, 0.8\}$ as the likelihood of edge presence (for random and bipartite graphs) or edge existence when the Euclidean distance between nodes is less than or equal to $\eta$ (for geometric graphs). Instances (d) refer to the complements of the bipartite graphs specified in (c). Due to the randomness inherent in the instance generator, each instance group consists of five instances, where a group corresponds to a fusion of graph class, number of vertices, and Euclidean distance (or probability for random and bipartite graphs). The groups are denoted by $C\_|V|\_\eta$, where $C$ denotes the graph class: random (rand), geometric (geo), bipartite (bip), and complement of bipartite (cbip). The results are consolidated for each instance group and presented as the mean of the five instances within it. Notably, each graph class entails 80 instances. The bipartite graph class was excluded from the experiments on the connected Grundy coloring problem because it is a trivial case where $\Gamma_c(G)$ is always 2. However, this class was included in the experiments for the Grundy coloring problem as it is NP-hard.

Instances (e) are defined as a subset of those with up to 500 vertices from the Second DIMACS Implementation Challenge. This set comprises 42 instances with $|V| \in [28,500]$. The attributes of these instances (number of vertices and density) are shown in Table~\ref{tab:dimacs_instances} (Appendix \ref{sec:dimacs_instance}). These instances are extensively employed in literature, particularly for coloring and maximum clique problem~\citep{AvHeZu03, LuHa10, MoGo18, NoPiSu18, SaCoFuLj19, MelQueSan21}.

\subsubsection{Connecting disconnected graphs}
It is necessary to point out that not all instances of this set are connected. Consequently, the disconnected instances were made connected using an algorithm that aims to maintain the connected Grundy number of the resulting graph as close as possible to that of the original graph's connected component with the highest connected Grundy number. The vertex with the highest degree with the lowest index of each connected component is selected and a path is created between these vertices from the lowest index to the one with the highest index between them, making the instance connected. Figure~\ref{fig:connect_example} illustrates the connection behavior of a graph $G$ formed by the two components $H$ and $F$, showing that for most cases $\Gamma_c(G) = \max(\Gamma_c(H), \Gamma_c(F))$. Figure~\ref{fig:connect_example_a} illustrates the case in which the connected vertex $e$ received a color greater than one, as the coloring is connected, so $g$ comes after $e$ in the order $\sigma_c$ and necessarily receives the color one. This implies that if the two components were not connected it would still be possible to achieve the same coloring for each component. Figure~\ref{fig:connect_example_b} represents the another possible case, but it can be noted that as $g$ receives a color greater than one, then necessarily one of its neighbors will have to receive color one, and if there is no edge $eg$, so the coloring of the component containing $g$ can also be achieved by starting from the neighbor of $g$. The process can be repeated for each pair of components as long as it is guaranteed that a path is built between the vertices that will be connected, and for most of these cases $\Gamma_c(G) = \max(\Gamma_c(H), \Gamma_c(F))$. An example in which this is not valid is a graph with two isolated vertices, where $\max(\Gamma_c(F),\Gamma_c(H))+1 = \Gamma_c(G) = 2 $.

 \begin{figure}[!ht]
    \centering
    \subfigure[Connecting from $e$, with \mbox{$c(e) > 1$}.]{
\begin{tikzpicture}
	[scale=1.1,every node/.style={circle,draw=black}]
	\node[label=b, fill=blue!30] (1) at  (0, 2)   {$1$};
	\node[label=d, fill=red!30] (3) at  (1.5, 1.5)   {$2$};
        \node[label=e, fill=green!30] (4) at  (2, 2.5)   {$3$};
        \node[label=f, fill=blue!30] (5) at  (3, 1)   {$1$};
        \node[label=g, fill=blue!30] (6) at  (3, 3)   {$1$};
        \node[label=h, fill=red!30] (7) at  (4, 3)   {$2$};
        \node[label=i, fill=red!30] (8) at  (4, 4.5)   {$2$};
    \begin{scope}[line width=1.6pt,every node/.style={}]
	\draw (1) edge (3);
        \draw (1) edge (4);
        \draw (3) edge (4);
        \draw (3) edge (5);
        \draw (4) edge (5);
        \draw (6) edge (7);
        \draw (6) edge (8);
        \draw[red] (6) edge (4);
	\end{scope}
 \end{tikzpicture}
 \label{fig:connect_example_a}
 }
 \hspace{1.3cm}
 \subfigure[Connecting from $e$, with \mbox{$c(e) = 1$}.]{
\begin{tikzpicture}
	[scale=1.1,every node/.style={circle,draw=black}]
	\node[label=b, fill=green!30] (1) at  (0, 2)   {$3$};
	\node[label=d, fill=red!30] (3) at  (1.5, 1.5)   {$2$};
        \node[label=e, fill=blue!30] (4) at  (2, 2.5)   {$1$};
        \node[label=f, fill=green!30] (5) at  (3, 1)   {$3$};
        \node[label=g, fill=red!30] (6) at  (3, 3)   {$2$};
        \node[label=h, fill=blue!30] (7) at  (4, 3)   {$1$};
        \node[label=i, fill=blue!30] (8) at  (4, 4.5)   {$1$};
    \begin{scope}[line width=1.6pt,every node/.style={}]
	\draw (1) edge (3);
        \draw (1) edge (4);
        \draw (3) edge (4);
        \draw (3) edge (5);
        \draw (4) edge (5);
        \draw (6) edge (7);
        \draw (6) edge (8);
        \draw[red] (6) edge (4);
	\end{scope}
 \end{tikzpicture}
 \label{fig:connect_example_b}
 }
 \caption{Connecting two components using the highest degree vertex of each component, with the new edge depicted in red.}
    \label{fig:connect_example}
\end{figure}
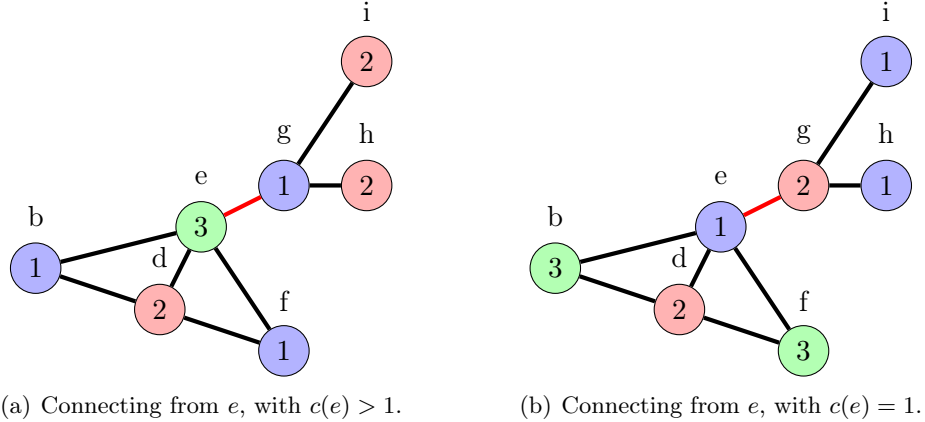

\subsection{Tested approaches and parameter settings}
\label{subsec:parameter}

The following approaches were considered in the computational experiments:
\begin{itemize}
    
    \item the standard formulation described in Section~\ref{sec:formgrundycoloring} (std);
    \item the formulation by representatives detailed in Section~\ref{sec:repformulations} (rep);
    \item the BRKGA with the restart and local search explained in Section \ref{sec:brkga} (BRKGA+R+LS);
    \item  a baseline BRKGA that does not employ restart or local search (BRKGA-B), whose working principle is similar to that of \cite{Silvaetal2024} for the Grundy coloring problem but with the decoder described in Algorithm \ref{alg:decoder_con_grundy}.
\end{itemize}

The solver was set with the default configurations and a single thread. A time limit of 3600 seconds (1 hour) was given for each formulation to solve each of the instances. A warm start (i.e., initial feasible solution) was provided to the formulations,  given by the best solution achieved using any of the following greedy heuristics:
\begin{itemize}
    \item \textit{connected minimum-degree first} (CMinDF), that defines a coloring sequence $(v_1,\ldots,v_n)$ prioritizing the vertices with lower degree described in Appendix \ref{sec:greedyheuristics}. 
    \item \textit{connected maximum-degree first} (CMDF), that defines a sequence $(v_1,\ldots,v_n)$ in which the highest degree vertex is iteratively chosen that is a neighbor of at least one previously chosen vertex;
    \item \textit{DSatur} \citep{Bre79}, that builds a sequence $(v_1,\ldots,v_n)$ using an adaptive criterion based on the maximum degree of saturation, where the degree of saturation of a vertex is equal to the number of vertices with different colors that are adjacent to it.
\end{itemize}

Notice that CMDF and DSatur are connected versions of some widely-used and well-established greedy criteria for the coloring problem (that aims to minimize the number of colors). However, it is known that they may sometimes produce colorings with a large number of colors.

The settings for the BRKGAs were defined as follows.
All executions were performed using a single thread.
A time limit of 300 seconds (five minutes) was defined as the stopping criterion instead of the number of generations. This choice makes it easier to control the total runtime and benefits efficient decoders.
The parameters were defined based on preliminary tests considering a subset of 28 instances and the following possible settings: $p$ = \{$1\times|V|$, $2\times|V|$, $3\times|V|$\}, $p_e$ = \{5\%, 15\%, 30\%\}, $p_m$ = \{5\%, 10\%, 30\%\} and $\rho_e$ = \{60\%, 70\%, 90\%\}. For the version with local search, the sizes considered for the population size were smaller to balance the time that will be spent on the local search, with the values considered $p$ = \{$1\times|V|$, $1.5\times| V|$, $1.7\times|V|$\}. The instances were chosen to be a representative sample, ensuring that at least five instances with varying densities and numbers of vertices were randomly selected from each class. We performed a grid search for parameter selection, i.e., all the possible parameter combinations were considered. Thus, in total, 81 configurations were tested for which the BRKGAs was run with five different seeds.
The considered values were based on how the metaheuristic converges over the generations by varying each of the parameters~\citep{gonccalves2011biased}.
Finally, the configuration selected for the overall experiments was $(p, p_e, p_m, \rho_e)$ = $(3\times|V|, 30\%, 10\%, 60\%)$ for BRKGA-B and $(p, p_e, p_m, \rho_e)$ = $(1.7\times|V|, 30\%, 10\%, 60\%)$ for BRKGA+R+LS.

Considering the BRKGA framework described in Section~\ref{sec:brkga}, we execute the BRKGA with reset and the local search described in Section~\ref{sec:NeighborhoodAlgorithm}. The parameter $g_{lim}$ was set to 2000 generations and the local search was applied to five solutions: the top-ranked solution with the highest fitness, along with four distinct solutions randomly chosen from the ELITE population in the generation where the best solution value improved. The parameter $g_{lim}$ was defined by analyzing the average number of generations required to improve the current best solution before the algorithm began to expend too many generations without further improvements. The criteria for selecting solutions for the local search was determined through a grid search conducted on the same sample of instances used for the other parameters.

\subsection{IP formulations results}
\label{subsec:ip_results}

Table \ref{tab:summary_ip} summarizes the results achieved by the standard formulation and the formulation by representatives for the instances with up to 30 vertices. The first column represents the graph class. The following columns provide, for each of the formulations, the average gap, average time, and total number of instances for which optimality was proven for each graph class. The penultimate row shows the average gap and time across all classes, while the last row displays the total number of instances for which optimality was achieved. More detailed results for the instances of classes (a), (b), and (d) with up to 30 vertices can be found in Appendix \ref{sec:results_ip_tiny}.


\begin{table}[!ht]
\small
\caption{Summary of the results using the IP formulations for instances with up to 30 vertices, separated by classes.}
\centering
\begin{tabular}{l|ccc|ccc}
\hline
\multicolumn{1}{c|}{} & 
\multicolumn{3}{c|}{std} & 
\multicolumn{3}{c}{rep} \\
class & gap(\%) & time(s) & \#opt & gap(\%) & time(s) & \#opt \\ 
\hline
 Geometric & 9.4 & 2455.8 & 27 & 14.3 & 2485.6 & 30\\
 Random & 33.9 & 3304.7 & 8 & 22.5 & 2553.4 & 29\\
 Complement & 10.1 & 1812.9 & 40 & 4.5 & 1373.1 & 53 \\
 \hline
Mean & 17.8 & 2524.4  & & 13.7 & 2137.3 \\
Total & & & 75 & & & 112  \\
\hline
\end{tabular}
\label{tab:summary_ip}
\end{table}

The results in Table~\ref{tab:summary_ip} show that the formulation by representatives outperformed the standard one. The graph class in which the formulations performed the best was the complement of bipartite graphs. This was an expected outcome for the formulation by representatives, as previously observed in the study by \citet{Silvaetal2024} on the Grundy coloring problem. Also, regarding the results presented in \citet{Silvaetal2024}, we can observe in our work a similar trend: the formulation by representatives performs better on denser instances and the standard formulation on sparser ones. However, the formulation by representatives performed better overall. This is evident not only in the geometric instances but especially in the random ones, where the standard formulation was able to prove far fewer optima. The most significant difference in favor of the formulation by representatives occurred in instances with 
$\eta = 0.8$, corresponding to a density close to 80\%, making them among the densest cases. 

It should be noted that $\Gamma_c(G) \leq \Gamma(G)$. Consequently, if for a given instance we know the value of the upper bound for $\Gamma(G)$ \citep{Silvaetal2024}, and we find a solution to the connected problem with the same value, then the solution is optimal for that instance. In total, the optimal solution was proved for 132 out of the 240 instances considering the two formulations. Moreover, as a result of this observation, it can be concluded that the optimal solution was found for a total of 178 instances considering the two formulations, which accounts for 74.16\% of the test cases with up to 30 vertices. 

The connected Grundy number problem proved challenging for both formulations. Experiments for instances with more than 50 vertices and from the DIMACS set using the default and single thread configurations did not generate satisfactory results. The solver only obtained a solution for 12\% of them. For the remaining cases, the process was killed by the computer due to memory overflow. Other attempts were made using multiple threads and changing the initial root relaxation method, but we could not achieve any significative improvement. 
Within these 12\% of instances for which an initial solution was obtained, the solver could only solve more than one node for a quarter of them. Considering this subset, we also increased the timeout to 10800 seconds (3h), but this led to a few more processes killed due to memory overflow. For these reasons, we do not perform further analysis of such experiments.

\subsection{BRKGA results}

In this section, we discuss the results achieved using the two BRKGA variants for the connected Grundy problem proposed in this work: BRKGA-B and BRKGA+R+LS. Tables \ref{tab:cbrkga_rand_rls}-\ref{tab:cbrkga_dimacs_rls} (see Appendix~\ref{sec:brkga_grundy_con}) present the results of the experiments. In the remainder of the paper, consider the percentage deviation between $val_1$ and $val_2$ to be $dev =100 \times (val_1 - val_2)/val_2$. 

\cite{Silvaetal2024} demonstrated that their BRKGA performs well in approximating $\Gamma(G)$. In the following, we will evaluate whether BRKGA-B can also approximate well the value $\Gamma_c(G)$. To do that, we compare the values obtained by BRKGA-B with those achieved using the BRKGA of \cite{Silvaetal2024} for the Grundy coloring problem. Thus, we can place it as a baseline when demonstrating that BRKGA-B can generate good results. Remember that $\Gamma_c(G) \leq \Gamma(G)$. Figure \ref{fig:conexo_geral} presents boxplots of the percentage deviations between the average solutions found by BRKGA-B for the connected Grundy coloring problem ($val_1$) and the average solutions obtained by \cite{Silvaetal2024}'s BRKGA for the Grundy coloring problem ($val_2$). As shown in Figure \ref{fig:conexo_geral}, the proposed BRKGA-B exhibits good performance, with the median of the percentage deviations close to 0, as well as narrow lower and upper bounds across these graph classes, despite tackling a more constrained problem.

\begin{figure}[!ht]
    \centering
    \includegraphics[width=0.5\linewidth]{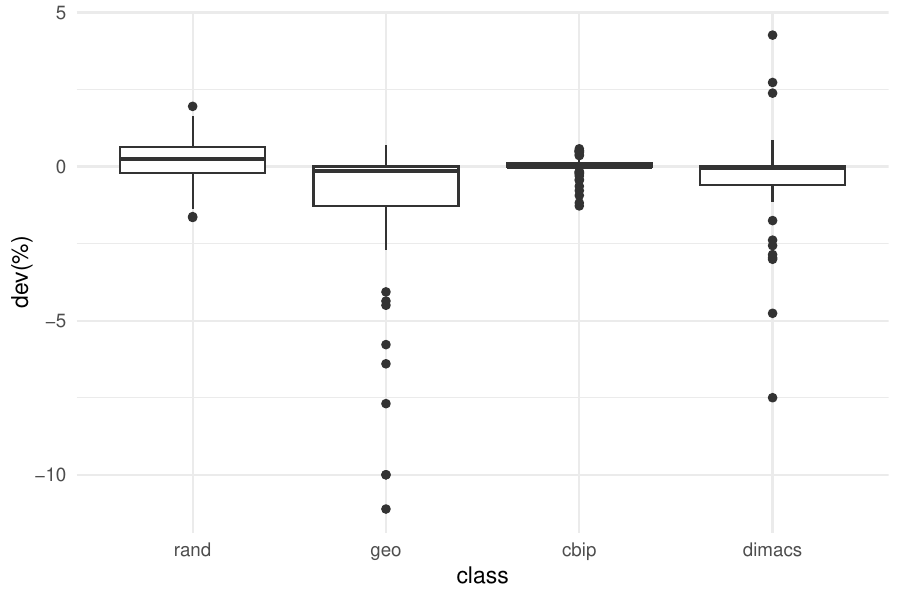}
    \caption{Boxplot of the percentage deviation of the average number of colors found by the \mbox{BRKGA-B} to the connected Grundy coloring problem to the average solution value of BRKGA for the Grundy coloring problem separated by class. All 282 instances tested by both methods are considered, consisting of all the DIMACS instances and those with at least 50 vertices of classes (a), (b), and (d).}
    \label{fig:conexo_geral}
\end{figure}

As previously discussed, BRKGA is a metaheuristic that finds good-quality solutions for the Grundy coloring problem \citep{Silvaetal2024}, and Figure \ref{fig:conexo_geral} shows that it also performs well for its connected version. However, it tends to converge quickly, with the class of graph requiring the longest average time to find the best solution taking less than 66 seconds. The reset mechanism and the local search procedure are strategies to overcome this issue: the reset initiates a new search, potentially exploring a different region, while the local search leverages part of the time to explore the neighborhood of a good solution in the current population trying to improve the best solution found. Table \ref{tab:summary_cbrkga} provides a summary of the results comparing BRKGA-B and BRKGA+R+LS, organized by graph class (first column). The second and third columns show the average percentage difference in the number of colors used for each class and the number of instances for which BRKGA+R+LS at least matched BRKGA-B. The fourth and fifth columns present similar information but focus on the maximum number of colors obtained. The last column indicates the number of instances for which BRKGA+R+LS outperformed BRKGA-B when considering the maximum number of colors.

\begin{table}[h]
\small
\caption{Summary of the results of BRKGA+R+LS compared with BRKGA-B, separated by classes.}
\centering
\begin{tabular}{l|cr|crc}
  \hline
class & $diff_{mean}$ & $\geq_{mean}$ & $diff_{max}$ & $\geq_{max}$ & $>_{max}$ \\ 
\hline
Random & 1.03 & 71/80 & 0.55 & 73/80 & 18/80\\
Geometric & 0.17 & 60/80 & 0.09 & 77/80 & 04/80\\
Complement & 0.10 & 63/80 & 0.00 & 80/80 & 00/80\\
DIMACS & 0.07 & 26/42 & 0.10 & 34/42 & 09/42\\
\hline
\end{tabular}
\label{tab:summary_cbrkga}
\end{table}

For the random instances, BRKGA+R+LS outperformed or matched BRKGA-B in 71 out of 80 cases regarding the average number of colors used, achieving an average improvement of 1.03\%. For the maximum number of colors, it achieved the same result or a better one in 73 cases, with an average improvement of 0.55\%. For the other instance classes, improvements in the maximum number of colors were smaller, with 0.09\% for the geometric instances and 0.10\% for the DIMACS instances. In these two classes, the same solution or a better one was found in 77 out of 80 and 36 out of 42 cases, respectively, despite the overall improvement being minor.

For the complement of bipartite graphs, no improvement was observed for the maximum number of colors when using BRKGA+R+LS. The same maximum number of colors was achieved for all the instances using BRKGA-B and BRKGA+R+LS, a predictable outcome given the high density of these graphs. It is also noticeable that, for the geometric and complement of bipartite graphs, BRKGA-B performed better for dense instances with at least 70 vertices. Out of the 282 instances tested, the same maximum was achieved in 234 cases, with differences in 48 cases, of which BRKGA+R+LS achieved improvements in 31 (Table~\ref{tab:summary_cbrkga}). In the following, we will provide a more detailed analysis of these differences regarding the maximum and average solutions. 

\begin{figure}
    \centering
    \subfigure[Boxplot of the percentage deviation of BRKGA+R+LS and BRKGA-B from the average solution value of BRKGA-B considering all the runs, separated by class.]{\includegraphics[width=0.42\linewidth]{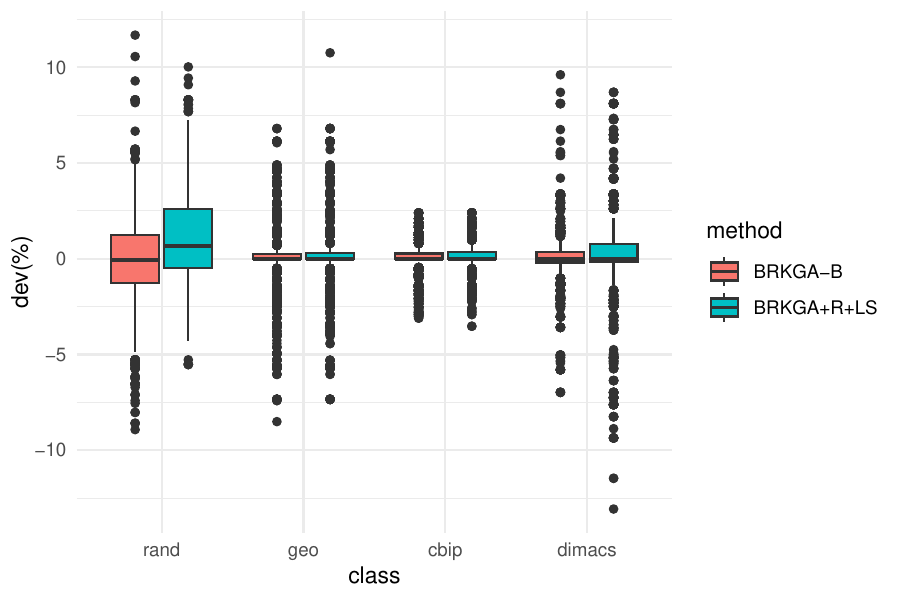}
    \label{fig:cbrkga_devmean}
    }
    \hspace{0.2cm}
    \subfigure[Boxplot of the percentage deviation of the best solution from BRKGA+R+LS and BRKGA-B where the best solutions value differ between the approaches, separated by class.]{\includegraphics[width=0.42\linewidth]{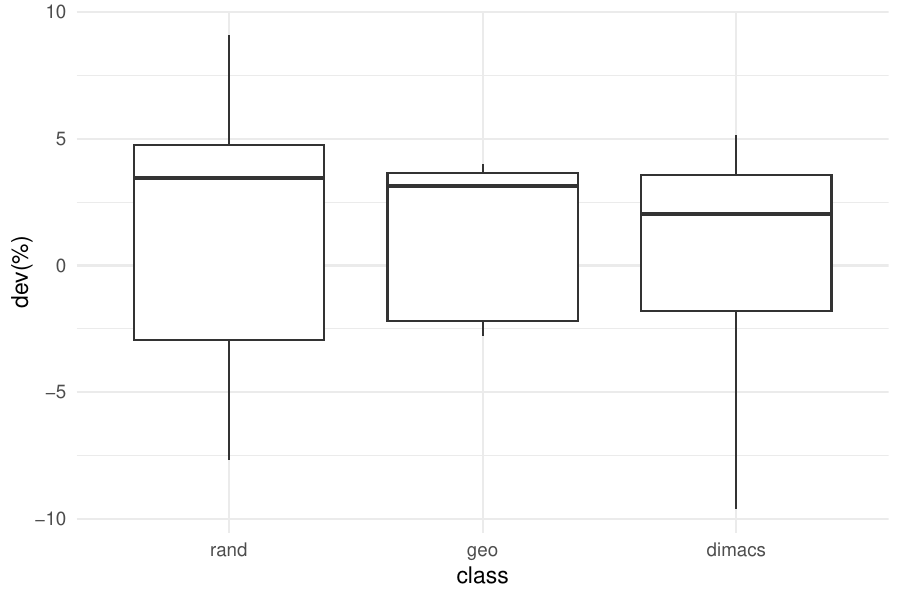}
    \label{fig:cbrkga_devmax}
    }
    \caption{Boxplots comparing the percentage deviations between BRKGA+R+LS and BRKGA-B. All the 282 instances tested by both methods are considered, consisting of all the DIMACS instances and those with at least 50 vertices from classes (a), (b), and (d).}
    \label{fig:cbrkga_dev}
\end{figure}

Figure \ref{fig:cbrkga_devmean} shows the deviations between the solutions achieved by BRKGA+R+LS and BRKGA-B for each execution of a instance ($val_1$) and the average of all the runs of BRKGA-B ($val_2$) for that instance. From Figure \ref{fig:cbrkga_devmean}, we can observe that the median deviation for every boxplot is close to 0\%, with a small interquartile range also near zero, particularly for the geometric, complement of bipartite, and DIMACS instances. For the random instances, the range is larger, with a greater proportion of positive values, especially for BRKGA+R+LS. This trend can also be observed for the DIMACS instances. This suggests that a solution obtained by BRKGA+R+LS in a single execution tends to be at least as good as the average solutions of BRKGA-B. The results indicate that, for isolated runs, it is preferable to choose BRKGA with reset and local search, particularly for random graphs. Additionally, a significant number of outliers can be found in some classes, but as previously noted, when considering the average of both methods, they tend to yield similar results most of the time, with the BRKGA+R+LS having some advantage. 

Figure \ref{fig:cbrkga_devmax} shows the deviations between the best solution obtained by BRKGA+R+LS for a given instance ($val_1$) with the best solution achieved by BRKGA-B ($val_2$). For a clearer visualization of distinct cases, all instances for which both methods yielded the same best solution value were discarded, resulting in 48 out of the 282 instances (25 random, 7 geometric, and 16 DIMACS). Thus, the complement of bipartite class does not appear in the graph. The boxplots show that the median is positive across the classes and closer to the third quartile, favoring BRKGA+R+LS. However, it is noteworthy that the first quartile remains negative. The random instances exhibited the greatest variation between minimum and maximum values, while the geometric instances had the least variation and the DIMACS instances had the lowest minimum. Nonetheless, DIMACS presented the smallest interquartile range, with a positive median. Considering these DIMACS instances, the mean deviation was also positive at 0.28\%.
\section{Adapting the local-search enhanced BRKGA to tackle the Grundy coloring problem}
\label{sec:adaptinggrundy}

In this section, we propose an approach to enhance the results of the Grundy coloring problem using the BRKGA with the decoder presented by \citet{Silvaetal2024}. 
We use the decoder from the literature but with the framework described in Section~\ref{sec:brkga} and the local search presented in Section~\ref{sec:NeighborhoodAlgorithm}. It is worth mentioning that Propositions \ref{lemma:worstcase}, \ref{lemma:intervalEqualLeft}, \ref{lemma:intervalEqualRight}, \ref{lemma:criticalposition}, and Corollaries \ref{col:minequalcolors}, \ref{col:maxdiffcolors}, \ref{col:criticalpositions}, do not depend on the connectivity constraint and are also valid for the Grundy coloring problem. The difference in $\mathcal{N}(S)$ for this problem is that any move operation generates a new valid sequence, so for any graph $|\mathcal{N}(S)|= O(V^2)$, if we do not consider any of the prunings resulting from the theoretical results presented in Section \ref{sec:neighborhoodanalysis}.

\begin{itemize}
    \item \textbf{BRKGA} + \textbf{R}eset + \textbf{L}ocal \textbf{S}earch (BRKGA+R+LS): The only difference in the local search algorithm (Section \ref{sec:NeighborhoodAlgorithm}) is that there is no connectivity check when performing a move apart from the checks made through the \textit{canMoveLeft} and \textit{canMoveRight}. We conducted the local search on five solutions: the top-ranked solution with the highest fitness, along with four distinct solutions randomly chosen from the ELITE population in the generation in which the best solution value improved. The other parameters used were the same as those used by \citet{Silvaetal2024}, except for the population size that was reduced to allow the BRKGA to run for more generations to compensate for the time spent on local search. Therefore, the parameters were $(p, p_e, p_m, \rho_e)$ = $(1.7\times|V|, 30\%, 10\%, 60\%)$. The value $g_{lim}$ was set to 2000 generations.\label{item:oconfig}
\end{itemize}

\subsection{Experiments and results}

In this section we summarize the results of the experiments with BRKGA+R+LS for the Grundy coloring problem and make a comparative analysis with the results obtained by the BRKGA of \cite{Silvaetal2024}, hereafter identified as BRKGA-G. We conducted the experiments in the same manner as performed by \cite{Silvaetal2024}. Each approach was run 50 independent executions, given a time limit of 300 seconds per run, with the other parameters configurations described in introduction of Section~\ref{item:oconfig}. Table \ref{tab:summary_brkga} presents a summary of the results comparing BRKGA+R+LS with \mbox{BRKGA-G}, separated by graph class (first column).
The second and third columns present the average percentage difference in the number of colors used for each class and the number of instances in which BRKGA+R+LS was better than or equal to BRKGA-G. The fourth and fifth columns are analogous, but refer to the maximum colors obtained. The last column indicates the number of instances for which BRKGA+R+LS outperformed BRKGA-G regarding the maximum number of colors.
More detailed results are presented in Appendix~\ref{sec:brkga_results_grundy}.  

 \begin{table}[!ht]
\small
\caption{Summary of the results of the BRKGA+R+LS compared with BRKGA-G (literature), separated by classes for the Grundy coloring problem.}
\centering
\begin{tabular}{l|cr|crc}
  \hline
class & $\mathit{diff}_{mean}$ & $\geq_{mean}$ & $\mathit{diff}_{max}$ & $\geq_{max}$ & $>_{max}$ \\ 
\hline
Random & 1.32 & 77/80 & 0.31 & 76/80 & 10/80\\
Geometric & 0.23 & 71/80 & 0.17 & 79/80 & 06/80 \\
Bipartite & 2.38 & 80/80 & 1.04 & 79/80 & 12/80\\
Complement & 0.17 & 73/80 & 0.00 & 80/80 & 00/80\\
DIMACS & 0.04 & 32/42 & 0.50 & 37/42 & 08/42\\
\hline
\end{tabular}
\label{tab:summary_brkga}
\end{table}
 
For the random graphs, the new approach improved the average number of colors obtained in each group. Regarding the maximum number of colors achieved, BRKGA+R+LS matched or improved the average and maximum number of colors in 77 and 76 out of 80 instances, respectively, only producing worse solutions for the group \textit{rand\_80\_0.4} on the average number of colors. For the geometric class, the average number of colors was better than or equal to the one available in the literature for 71 out of 80 instances, having an average improvement of 0.23\%. For the maximum number of colors, at least matching the results in 79 out of 80 cases and having an average improvement of 0.17\%. For the bipartite graphs, the BRKGA+R+LS approach also improved the average number of colors across all groups, being better than or equal to for all instances, having a improvement of 2.38\% in this class. In terms of the maximum number of colors, it was better than or equal to in 79 out of 80 instances, with a 1.04\% improvement in the average maximum number of colors for this instances. For the complement of bipartite graphs, considering the average number of colors, it managed to be better than or equal to in 73 of the 80 cases with a small percentage increase, while for the maximum it achieved the same values for all cases. However, this result was expected, given that these are very dense instances, where any feasible solution tends to use a high number of colors. Therefore, even small percentage increases are significant, and improvements in the averages were observed in 14 out of 16 groups (Table \ref{tab:brkga_cbip_rls}, Appendix \ref{sec:brkga_results_grundy}).

The DIMACS instances exhibit the greatest variability in the quality of results for the new approach (Table \ref{tab:brkga_dimacs_rls}, Appendix \ref{sec:brkga_results_grundy}). Considering the average number of colors, it improved in 19 cases, and better than or equal to for 37 of 42 instances, but in the instance \textit{hamming8-2}, it reached a $\mathit{diff}_m$ of -7.69\%. This instance has 256 vertices and a very high density of 0.96 (almost that of a complete graph), which caused the execution time to be primarily consumed during the first local search. This is due to the total number of neighbors being $O(256^2)$, and the pruning techniques used to optimize the neighborhood size were not highly effective, similar to the coloring of each neighbor. Nevertheless, the overall average for the set remained positive with 0.04\% of improvement. Regarding the maximum number of colors, there was also significant variability, with improvements in some cases and worse results in others. The $\mathit{diff}_x$ ranged from -6.21\% to 7.69\%, with the final average for the maximum number of colors being 0.50\%. Considering all instances, a new best solution was found in 36 cases, having an overall average increase in maximum colors obtained of 0.39\% and an average increase of 0.91\% in average colors across all runs.

Table \ref{tab:summary_brkga} shows that most of the improvements in the best solution were found in the random and bipartite graph classes, with 10 and 12 instances, respectively. Additionally, better solutions were found regarding the maximum number of colors in six geometric and eight DIMACS instances. Overall, new best solutions were achieved for 36 out of the 362 instances, representing an improvement in 9.9\% of the cases.

Figure \ref{fig:brkga_boxplot_dev} considers the deviation between the solutions achieved using BRKGA-G and BRKGA+R+LS for each execution of an instance, and the average of all the runs of the BRKGA-G for the same instances provided in \cite{Silvaetal2024}. Thus, the deviation was computed for every execution of the BRKGA+R+LS. From Figure \ref{fig:brkga_boxplot_dev}, it can be stated that the new approach consistently outperforms the literature for the random and bipartite graph classes where the boxplots are shifted up. For the other classes, there is no clear trend of surpassing the average solution over multiple runs. In fact, the results suggest that, in the median case, the solutions will be equivalent, with occasional instances where the solutions may be either better or worse. Figure \ref{fig:brkga_boxplot_dev_0} compares the deviation between the best solution found by BRKGA+R+LS and that obtained by BRKGA-G for each instance. To facilitate the visualization of distinct cases, all instances for which both methods achieved the same solution were removed, leaving 47 instances (14 random, 7 geometric, 13 bipartite, and 13 DIMACS). Consequently, the complement of bipartite class does not appear in the figure, as the maximum number of colors found by both approaches was identical in all cases. Considering only the distinct cases, Figure \ref{fig:brkga_boxplot_dev_0} reveals favorable performance by BRKGA+R+LS, with positive medians across all presented cases. Additionally, since the analysis does not focus on individual instances, it can be observed that most of these cases show substantial improvements, with a median close to 2.5\% for the DIMACS instances, above 2.5\% for the random ones, and over 5.0\% for the bipartite graphs. In general, when the best solutions obtained by each method differed, BRKGA+R+LS outperformed BRKGA-G.

\begin{figure}[!ht]
    \centering
    \subfigure[Boxplot of the percentage deviation of BRKGA+R+LS and BRKGA-B from the average solution value of BRKGA-G considering all the runs, separated by class.]{
    \includegraphics[width=0.43\linewidth]{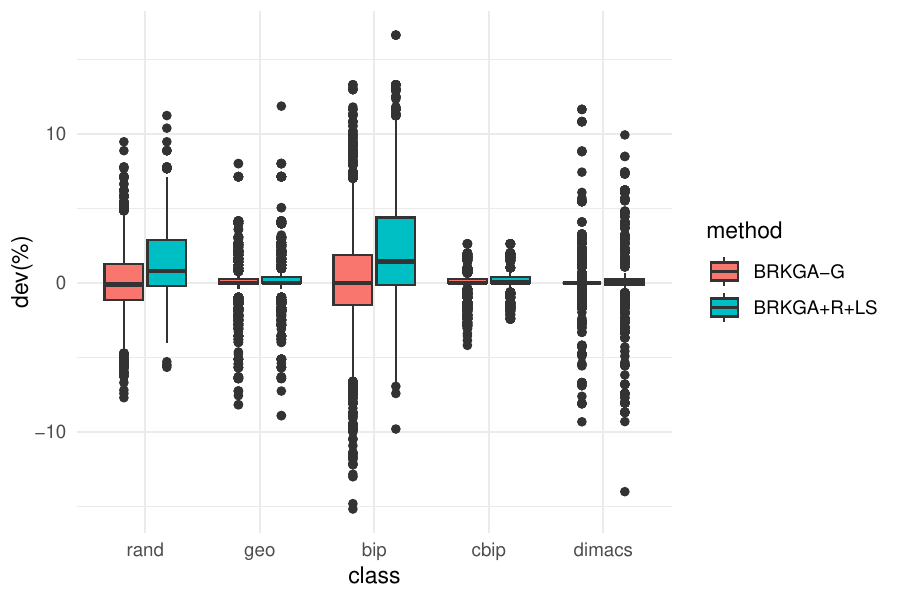}
    \label{fig:brkga_boxplot_dev}
    }
    \hspace{0.2cm}
    \subfigure[Boxplot of the percentage deviation of the best solution from BRKGA+R+LS and BRKGA-G where the best solutions value differ between the approaches, separated by class.]{
    \includegraphics[width=0.43\linewidth]{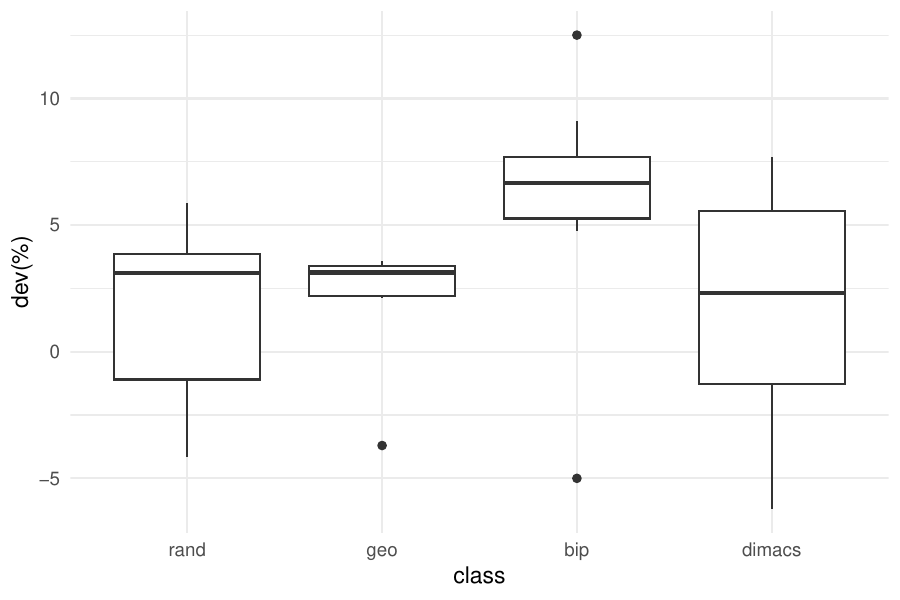}
    \label{fig:brkga_boxplot_dev_0}
    }
    \caption{Boxplots comparing the percentage deviations between BRKGA+R+LS and BRKGA-G solutions (literature) for the Grundy coloring problem. All the 362 instances tested by both methods are considered, consisting of all the DIMACS instances and those from classes (a)-(d) with at least 50 vertices.}
\end{figure}

\section{Concluding remarks}
\label{sec:conclusions}

In this paper, we considered the connected Grundy coloring problem.
We presented the first integer programming formulations for the problem: a standard formulation and a formulation by representatives. Additionally, we proposed a local-search enhanced BRKGA.
To the best of our knowledge, our approaches are the first general optimization methods for the connected Grundy coloring problem.
Moreover, we performed a theoretical analysis of the neighborhood employed in the local search procedure.
We carried out extensive experiments to evaluate the performance of the approaches.

The computational results indicate that the integer programming formulations are suitable for instances with up to 30 vertices.  The formulation by representatives performs better than the standard formulation, particularly for denser instances. We remark that a similar behavior was observed for the Grundy coloring problem \citep{Silvaetal2024}. Additionally, in graphs with up to 30 vertices, 178 optima were found out of the 240 tested instances (74.16\%), although only 132 were directly solved to optimality by the formulations. This suggests that the formulations effectively find high-quality solutions, and even optima ones, but they face challenges to prove optimality.

The tests for the connected Grundy problem using BRKGA-B and BRKGA+R+LS demonstrated that these approaches perform well, achieving good-quality solutions. They were used to successfully tackle instances with up to 500 vertices. The results indicate that BRKGA+R+LS generally outperforms BRKGA-B in terms of the average total number of colors obtained, particularly for random graphs. The two BRKGA variants reached the same best solutions for most of the instances, except in 48 cases. Out of these, BRKGA+R+LS achieved the best results for 31, where the reset and local search address fast convergence, facilitating a broader exploration of the search space and leading to equal or improved solutions. 

Our study also demonstrated that combining a reset mechanism and local search with the BRKGA is effective for the Grundy coloring problem, achieving new best-known solutions for 36 instances. Additionally, BRKGA+R+LS exhibited robust performance, surpassing BRKGA-G in both average and best solutions, particularly for the random and bipartite graph classes. Overall, the results emphasize the importance of combining reset and local search to address the BRKGA's fast convergence for these problems, enabling a broader exploration of the search space. 

\vspace{0.8cm}

{
\noindent \small 
\textbf{Acknowledgments:} 
This study was financed in part by the Coordenação de Aperfeiçoamento de Pessoal de Nível Superior – Brasil (CAPES) – Finance Code 001.
This work was partially supported by the Brazilian National Council for Scientific and Technological Development (CNPq) research grant 314718/2023-0 and the FAPESB INCITE PIE0002/2022 grant.
}

\bibliographystyle{abbrvnat}
\bibliography{99_main}


\newpage

\begin{appendices}
\section{An adaptive greedy heuristic}
\label{sec:greedyheuristics}

In this section, we propose an BFS-based adaptive criterion for obtaining connected Grundy colorings, denoted the connected smallest-degree first. 
It defines an order by giving priority to the lowest-degree vertex in the graph induced by the vertices that did not yet receive a color, in an expansion based on a BFS.
The criterion is adaptative as each new color assignment modifies the graph used to make the next greedy choice.

The \textit{connected smallest-degree first} defines a sequence $(v_1,\ldots,v_n)$ for the connected Grundy coloring problem as described in Algorithm~\ref{alg:connectedsmallestdegreefirst}. Notice that the algorithm assumes a connected graph. Nevertheless, it can be easily modified to deal with graphs containing more than one connected component. Line \ref{l:c01}-\ref{l:c03} initialize every vertex as not in sequence with color 0, an auxiliary set $Q$ that stores the next possible vertices for a connected sequence and $H$ to be an auxiliary graph. Lines \ref{l:c04}-\ref{l:c07} select the first vertex, insert to $Q$, and sign the vertex as 1. The loop \ref{l:c08}-\ref{l:c16} selects the vertex $v_j$ in the set $Q$ with the lowest degree in $H$, adjacent to at least one vertex in the sequence, and removes it from both the set and the auxiliary graph. The inner loop \ref{l:c10}-\ref{l:c13} adds previously uninserted adjacent vertices of $v_j$ to set $Q$. Finally, in line \ref{l:c17} return a connected sequence.

\begin{algorithm}[H]
\small
\caption {Connected smallest-degree first ($G$)}
\label{alg:connectedsmallestdegreefirst}
    $reached_v \leftarrow 0$ for every $v\in V$\;\label{l:c01}
    $Q \leftarrow \varnothing$\;\label{l:c02}
    $H \leftarrow G$\;\label{l:c03}
    $v_1 \leftarrow$ vertex $u$ with minimum degree in $H$\; \label{l:c04}
    Add $v_1$ to $Q$\; \label{l:c05}
    $reached_{v_1} \leftarrow 1$\;\label{l:c06}
    $j\leftarrow 1$\;\label{l:c07}
    \While{$Q\neq \varnothing$}{ \label{l:c08}
        $v_j \leftarrow$ vertex $u$ in $Q$ with minimum degree in $H$\;\label{l:c09}
        \ForEach{neighbor $u$ of $v_j$ in $H$}{\label{l:c10}
            \If{$reached_u = 0$}{\label{l:c11}
                Add $u$ to $Q$\;\label{l:c12}
                $reached_u \leftarrow 1$\;\label{l:c13}
            }
        
        }
        $H \leftarrow H - v_{j}$\;\label{l:c14}
        Remove $v_j$ from $Q$\;\label{l:c15}
        $j \leftarrow j+1$\;\label{l:c16}
    }
    \Return the sequence $(v_1,v_2,\ldots,v_n)$\; \label{l:c17}
\end{algorithm}

\clearpage

\section{IP results for the smaller instances with up to 30 vertices}
\label{sec:results_ip_tiny}

Tables \ref{tab:grundy_geo_graph_leq_30}-\ref{tab:grundy_cbip_graph_leq_30} 
summarize the results using the formulations to tackle smaller instances with up to 30 vertices (see Section~\ref{subsec:ip_results}). In these tables, the first column represents the instance group so that each row corresponds to the average values over its five instances. The second column (cub) provides the best combinatorial upper bound, considering $\zeta(G)$, $\Delta_2(G)+1$, and $\Psi(G)$. The third column (h) indicates the average value of the initial heuristic solution value provided for the formulation. In the following, for each of the formulations, the columns indicate the average of the best-obtained solutions (best), the average of the dual bound achieved by the solver at the end of the execution (ub), the average optimality \textit{gap} in percent, calculated for each individual instance as $\frac{ub - best}{best}$ (notice that here we refer to the best solution and the bound for a specific instance, not the average), the average execution time (time), and the number of instances solved to optimality (\#opt).
The last two rows of the tables indicate the average of the values across all rows and the total number of instances solved to optimality. The largest values in the columns best are highlighted in bold.

\begin{table}[!ht]
\centering
\footnotesize
\captionsetup{font=small}
\caption{Results using the formulations for the geometric graphs with at most 30 vertices.}
\begin{tabular}{lrr|rrrrr|rrrrr}
  \hline
  \multicolumn{3}{c|}{} &
  \multicolumn{5}{c|}{std} & 
  \multicolumn{5}{c}{rep} \\
group & cub & h & best & ub & gap & time & \#opt & best & ub & gap & time & \#opt \\ 
     \hline
geo\_15\_0.2 &  4.4 &  \textbf{3.6} &  \textbf{3.6} &  3.6 &  0.0 &    2.0 & 5 &  \textbf{3.6} &  3.8 &  6.7 & 1291.6 & 4 \\ 
  geo\_15\_0.4 &  7.5 &  6.0 &  \textbf{6.8} &  7.0 &  3.6 & 1137.0 & 3 &  \textbf{6.8} &  7.2 &  7.2 & 2464.6 & 2 \\ 
  geo\_15\_0.6 & 10.8 &  9.6 & \textbf{10.2} & 10.8 &  6.0 & 2211.7 & 2 & \textbf{10.2} & 10.6 &  4.2 & 1449.4 & 3 \\ 
  geo\_15\_0.8 & 12.8 & 12.0 & \textbf{12.4} & 12.8 &  3.5 & 1445.2 & 3 & \textbf{12.4} & 12.4 &  0.0 &    1.0 & 5 \\ 
  geo\_20\_0.2 &  4.4 &  3.6 &  \textbf{3.8} &  3.8 &  0.0 &   37.9 & 5 &  3.6 &  4.8 & 33.3 & 2880.3 & 2 \\ 
  geo\_20\_0.4 & 11.2 &  8.2 &  9.2 & 11.2 & 22.0 & 3600.0 & 0 &  \textbf{9.6} & 11.2 & 17.3 & 3600.0 & 0 \\ 
  geo\_20\_0.6 & 13.2 & 10.6 & 12.0 & 13.2 & 10.2 & 3600.0 & 0 & \textbf{12.2} & 13.0 &  6.8 & 3027.8 & 1 \\ 
  geo\_20\_0.8 & 16.6 & 15.0 & 15.6 & 16.6 &  6.6 & 2886.5 & 1 & \textbf{16.0} & 16.0 &  0.0 &   30.9 & 5 \\ 
  geo\_25\_0.2 &  5.0 &  4.4 &  \textbf{4.6} &  4.6 &  0.0 &   110.2 & 5 & \textbf{4.6} &  6.0 & 32.9 & 3600.0 & 0 \\ 
  geo\_25\_0.4 & 12.6 &  9.4 & \textbf{10.6} & 12.6 & 20.0 & 3600.0 & 0 & 10.2 & 12.6 & 24.4 & 3600.1 & 0 \\ 
  geo\_25\_0.6 & 17.6 & 14.4 & \textbf{15.8} & 17.6 & 11.5 & 3600.1 & 0 & 15.6 & 17.6 & 13.0 & 3600.1 & 0 \\ 
  geo\_25\_0.8 & 20.8 & 18.0 & 19.4 & 20.8 &  7.3 & 3600.1 & 0 & \textbf{19.6} & 19.8 &  1.1 &  998.3 & 4 \\ 
  geo\_30\_0.2 &  6.8 &  5.0 & \textbf{5.6} &  6.2 & 10.7 & 2509.8 & 2 &  5.4 &  7.0 & 32.3 & 3600.1 & 0 \\ 
  geo\_30\_0.4 & 15.2 & 10.8 & \textbf{12.6} & 15.2 & 21.6 & 3600.1 & 0 & 12.0 & 15.2 & 27.7 & 3600.2 & 0 \\ 
  geo\_30\_0.6 & 22.2 & 17.4 & \textbf{18.8} & 22.2 & 18.5 & 3600.3 & 0 & 18.4 & 22.0 & 20.0 & 3600.1 & 0 \\ 
  geo\_30\_0.8 & 25.0 & 21.6 & 22.8 & 25.0 &  9.9 & 3600.5 & 0 & \textbf{23.4} & 23.8 &  1.9 & 2198.9 & 4 \\ 
   \hline
   Mean & 12.8 & 10.6 & \textbf{11.4} & 12.7 & 9.4 & 2455.8 &  & \textbf{11.4} & 12.7 & 14.3 & 2485.6 & \\
   Total &  &  &  & & & & 27 & & &  &  & \textbf{30}  \\
   \hline  
\end{tabular}
\label{tab:grundy_geo_graph_leq_30}
\end{table}

\begin{table}[!ht]
\centering
\footnotesize
\captionsetup{font=small}
\caption{Results using the formulations for the random graphs with at most 30 vertices.}
\begin{tabular}{lrr|rrrrr|rrrrr}
  \hline
  \multicolumn{3}{c|}{} &
  \multicolumn{5}{c|}{std} & 
  \multicolumn{5}{c}{rep} \\
group & cub & h & best & ub & gap & time & \#opt & best & ub & gap & time & \#opt \\ 
   \hline
rand\_15\_0.2 &  5.4 &  3.8 &  \textbf{4.2} &  4.2 &  0.0 &   24.1 & 5 &  \textbf{4.2} &  4.6 &  9.0 & 2102.4 & 3 \\ 
  rand\_15\_0.4 &  8.4 &  5.4 &  \textbf{6.6} &  7.6 & 14.3 & 3159.9 & 1 &  \textbf{6.6} &  7.0 &  6.9 & 1983.6 & 3 \\ 
  rand\_15\_0.6 & 10.4 &  6.6 &  \textbf{8.8} & 10.4 & 18.0 & 3600.0 & 0 &  \textbf{8.8} &  9.0 &  2.5 &  856.5 & 4 \\ 
  rand\_15\_0.8 & 12.6 &  9.6 & \textbf{10.8} & 12.6 & 16.9 & 3600.0 & 0 & \textbf{10.8} & 10.8 &  0.0 &    5.3 & 5 \\ 
  rand\_20\_0.2 &  6.8 &  4.2 & \textbf{5.4} &  6.4 & 17.3 & 2975.3 & 1 &  5.2 &  6.8 & 31.0 & 3600.0 & 0 \\ 
  rand\_20\_0.4 & 11.8 &  6.8 &  8.4 & 10.8 & 28.3 & 3600.0 & 0 & \textbf{8.6} & 10.8 & 25.5 & 3600.0 & 0 \\ 
  rand\_20\_0.6 & 14.0 &  9.2 & 10.2 & 14.0 & 37.3 & 3600.0 & 0 & \textbf{11.0} & 11.6 &  5.8 & 2486.0 & 3 \\ 
  rand\_20\_0.8 & 17.2 & 11.4 & 12.8 & 17.2 & 34.7 & 3600.0 & 0 & \textbf{13.2} & 13.2 &  0.0 &   79.8 & 5 \\ 
  rand\_25\_0.2 &  9.4 &  5.4 & \textbf{6.8} &  8.4 & 23.8 & 3516.6 & 1 & \textbf{6.8} &  9.4 & 38.6 & 3600.0 & 0 \\ 
  rand\_25\_0.4 & 13.8 &  7.6 & \textbf{8.8} & 13.8 & 56.7 & 3600.0 & 0 & \textbf{8.8} & 13.6 & 54.7 & 3600.0 & 0 \\ 
  rand\_25\_0.6 & 18.0 & 10.6 & 11.8 & 18.0 & 53.1 & 3600.1 & 0 & \textbf{12.6} & 15.6 & 24.6 & 3600.0 & 0 \\ 
  rand\_25\_0.8 & 21.6 & 13.4 & 15.2 & 21.6 & 42.3 & 3600.1 & 0 & \textbf{16.6} & 16.6 &  0.0 &  978.6 & 5 \\ 
  rand\_30\_0.2 &  9.6 &  5.8 & \textbf{6.8} &  9.0 & 32.4 & 3600.0 & 0 &  6.6 &  9.6 & 46.2 & 3600.2 & 0 \\ 
  rand\_30\_0.4 & 16.0 &  8.2 & \textbf{9.6} & 16.0 & 66.9 & 3600.1 & 0 & \textbf{9.6} & 16.0 & 67.1 & 3600.2 & 0 \\ 
  rand\_30\_0.6 & 21.4 & 11.8 & \textbf{13.8} & 21.4 & 56.0 & 3600.2 & 0 & \textbf{13.8} & 19.0 & 38.1 & 3600.2 & 0 \\ 
  rand\_30\_0.8 & 25.6 & 15.2 & 17.6 & 25.6 & 45.7 & 3600.3 & 0 & \textbf{18.4} & 20.2 & 10.1 & 3561.6 & 1 \\ 
   \hline
   Mean & 13.8 & 8.4 & 9.8 & 13.5 & 33.9 & 3304.7 &  & \textbf{10.1} & 12.1 & 22.5 & 2553.4 & \\
   Total &  &  &  & & & & 8 & & &  &  & \textbf{29}  \\
   \hline
\end{tabular}
\label{tab:grundy_rand_graph_leq_30}
\end{table}

\begin{table}[!ht]
\centering
\footnotesize
\captionsetup{font=small}
\caption{Results using the formulations for the complements of bipartite graphs with at most 30 vertices.}
\begin{tabular}{lrr|rrrrr|rrrrr}
  \hline
  \multicolumn{3}{c|}{} &
  \multicolumn{5}{c|}{std} & 
  \multicolumn{5}{c}{rep} \\
group & cub & h & best & ub & gap & time & \#opt & best & ub & gap & time & \#opt \\ 
   \hline
cbip\_15\_0.2 & 12.6 &  9.8 & \textbf{10.8} & 12.6 & 16.9 & 3600.0 & 0 & \textbf{10.8} & 10.8 &  0.0 &   34.8 & 5 \\ 
  cbip\_15\_0.4 & 10.2 &  8.8 & \textbf{9.4} & 10.2 &  8.7 & 2902.9 & 1 & \textbf{9.4} & 10.2 &  8.7 & 2932.5 & 1 \\ 
  cbip\_15\_0.6 &  9.2 &  \textbf{9.2} & \textbf{9.2} &  9.2 &  0.0 &    0.9 & 5 &  \textbf{9.2} &  9.2 &  0.0 &    0.7 & 5 \\ 
  cbip\_15\_0.8 &  8.6 &  \textbf{8.6} & \textbf{8.6} &  8.6 &  0.0 &    0.8 & 5 &  \textbf{8.6} &  8.6 &  0.0 &    0.5 & 5 \\ 
  cbip\_20\_0.2 & 17.4 & 12.6 & 13.8 & 17.4 & 26.2 & 3600.0 & 0 & \textbf{14.0} & 14.0 &  0.0 &  151.8 & 5 \\ 
  cbip\_20\_0.4 & 14.2 & 11.6 & \textbf{12.4} & 14.2 & 14.8 & 3600.0 & 0 & \textbf{12.4} & 14.0 & 13.1 & 3600.0 & 0 \\ 
  cbip\_20\_0.6 & 11.2 & \textbf{11.2} & \textbf{11.2} & 11.2 &  0.0 &   24.1 & 5 & \textbf{11.2} & 11.4 &  2.0 &  722.6 & 4 \\ 
  cbip\_20\_0.8 & 11.8 & \textbf{11.8} & \textbf{11.8} & 11.8 &  0.0 &   49.8 & 5 & \textbf{11.8} & 12.0 &  2.0 &  722.4 & 4 \\ 
  cbip\_25\_0.2 & 21.4 & 15.6 & 16.4 & 21.4 & 30.8 & 3600.1 & 0 & \textbf{17.2} & 18.0 &  4.7 & 2484.4 & 3 \\ 
  cbip\_25\_0.4 & 17.8 & 14.6 & 14.8 & 17.8 & 20.6 & 3600.1 & 0 & \textbf{15.0} & 17.8 & 19.1 & 3600.1 & 0 \\ 
  cbip\_25\_0.6 & 15.0 & \textbf{15.0} & \textbf{15.0} & 15.0 &  0.0 &    7.4 & 5 & \textbf{15.0} & 15.0 &  0.0 &   11.3 & 5 \\ 
  cbip\_25\_0.8 & 14.0 & \textbf{14.0} & \textbf{14.0} & 14.0 &  0.0 &    6.2 & 5 & \textbf{14.0} & 14.0 &  0.0 &   11.4 & 5 \\ 
  cbip\_30\_0.2 & 26.2 & 19.4 & 20.2 & 26.2 & 30.5 & 3600.4 & 0 & \textbf{20.6} & 22.2 &  8.4 & 3305.4 & 2 \\ 
  cbip\_30\_0.4 & 22.0 & 19.6 & \textbf{20.0} & 22.0 & 11.1 & 2903.9 & 1 & 19.8 & 22.0 & 12.7 & 2887.5 & 1 \\ 
  cbip\_30\_0.6 & 16.6 & \textbf{16.6} & \textbf{16.6} & 16.6 &  0.0 &   43.6 & 5 & \textbf{16.6} & 16.6 &  0.0 &   44.4 & 5 \\ 
  cbip\_30\_0.8 & 16.2 & \textbf{16.2} & \textbf{16.2} & 16.6 &  2.7 & 1467.1 & 3 & \textbf{16.2} & 16.6 &  2.7 & 1460.1 & 3 \\ 
   \hline
    Mean & 15.2 & 13.4 & 13.7 & 15.3 & 10.1 & 1812.9 &  & \textbf{13.8} & 14.5 & 4.5 & 1373.1 & \\
   Total &  &  &  & & & & 40 & & &  &  & \textbf{53}  \\
   \hline
\end{tabular}
\label{tab:grundy_cbip_graph_leq_30}
\end{table}
\clearpage
\section{BRKGA results for the connected Grundy coloring problem}
\label{sec:brkga_grundy_con}

The following Tables \ref{tab:cbrkga_rand_rls}-\ref{tab:cbrkga_dimacs_rls} presents the results of BRKGA+R+LS compared with the BRKGA-B approach presented in Section~\ref{subsec:decoder}. The first column represents an instance group, so all the values in each row correspond to the average over five instances, except for \textit{DIMACS} instances where each represents a single instance. Columns 2-4, 5-7 provide the average mean, maximum, and time to best (ttb) considering the 50 independent runs for each of the instance for the BRKGA and BRKGA+R+LS, respectively. Columns $diff_m = 100*\frac{mean_{\textrm{BRKGA+R+LS}}-mean_{\textrm{BRKGA-B}}}{mean_{\textrm{BRKGA-B}}}$ and $diff_x = 100*\frac{max_{\textrm{BRKGA+R+LS}}-max_{\textrm{BRKGA-B}}}{max_{\textrm{BRKGA-B}}}$, and represent the percentage improvement in relation to literature results for mean and best solution. Positive values for the $diff$ columns indicate better results from BRGKA+R+LS. The $\geq_{mean}$ and $\geq_{max}$ columns indicate how many times the new approach obtained better or equal solutions within each group about the average and maximum colors obtained by BRKGA-B, respectively. The last column ($>_{max}$) reports the quantity in which it was strictly better. The tables' last two rows indicate the average values across all rows and the total number of instances that the new approach was better than or equal to the literature results. Note that Table \ref{tab:brkga_dimacs_rls} corresponding to \textit{DIMACS} instances does not contain the $\geq_{mean}$, $\geq_{max}$ columns or the \textit{Total} row because each row corresponds to a single instance.

\begin{table}[ht]
\centering
\small
\caption{BRKGA-B and BRKGA+R+LS results for the random graphs.}
\begin{tabular}{l|rrr|rrrrrrrr}
  \hline
\multicolumn{1}{c|}{} & 
\multicolumn{3}{c|}{BRKGA-B} & 
\multicolumn{6}{c}{BRKGA+R+LS} \\
group & mean & max & ttb & mean & max & ttb & $diff_{m}$ & $\geq_{mean}$ & $diff_{x}$ & $\geq_{max}$ & $>_{max}$\\
\hline
rand\_50\_0.2 & 11.13 & 11.60 & 31.3 & 11.24 & 11.60 &  24.2 & 0.95 & 5 & 0.28 & 4 & 1\\ 
rand\_50\_0.4 & 16.20 & 17.20 & 43.0 & 16.57 & 17.20 &  80.1 & 2.25 & 5 & 0.07 & 4 & 1\\ 
rand\_50\_0.6 & 22.65 & 23.40 & 57.4 & 23.04 & 23.60 &  78.1 & 1.70 & 5 & 0.87 & 5 & 1\\ 
rand\_50\_0.8 & 30.72 & 31.40 & 53.4 & 30.90 & 31.40 &  62.7 & 0.59 & 5 &  0.00 & 5 & 0\\ 
rand\_60\_0.2 & 12.14 & 12.60 & 40.4 & 12.32 & 12.40 &  44.1 & 1.48 & 4 &  -1.54 & 4 & 0\\ 
rand\_60\_0.4 & 18.94 & 19.80 & 55.9 & 19.22 & 19.80 &  84.0 & 1.51 & 5 & 0.05 & 4 & 1\\ 
rand\_60\_0.6 & 25.70 & 26.80 & 64.7 & 26.08 & 27.00 &  91.2 & 1.49 & 5 & 0.77 & 5 & 1\\ 
rand\_60\_0.8 & 35.62 & 36.40 & 59.1 & 35.81 & 36.60 &  80.3 & 0.55 & 5 & 0.54 & 5 & 1\\ 
rand\_70\_0.2 & 13.24 & 13.60 & 25.8 & 13.43 & 14.00 &  56.9 & 1.44 & 5 & 3.08 & 5 & 2 \\ 
rand\_70\_0.4 & 20.92 & 21.80 & 51.9 & 21.14 & 22.00 &  81.1 & 1.05 & 5 & 0.95 & 5 & 1\\ 
rand\_70\_0.6 & 28.24 & 29.20 & 65.6 & 28.54 & 29.80 & 105.1 & 1.05 & 5 & 2.07 & 5 & 3\\ 
rand\_70\_0.8 & 39.46 & 40.60 & 54.5 & 39.59 & 40.80 & 106.2 & 0.34 & 3 & 0.50 & 5 & 1 \\ 
rand\_80\_0.2 & 14.42 & 15.00 & 42.7 & 14.66 & 15.00 &  85.9 & 1.64 & 5 & 0.08 & 4 & 1\\ 
rand\_80\_0.4 & 22.58 & 23.60 & 52.9 & 22.64 & 23.60 &  87.5 & 0.27 & 4 & 0.04 & 4 & 1\\ 
rand\_80\_0.6 & 31.65 & 32.80 & 65.6 & 31.73 & 33.00 & 103.2 & 0.25 & 3 & 0.64 & 4 & 2\\ 
rand\_80\_0.8 & 43.36 & 44.80 & 64.6 & 43.33 & 45.00 & 112.8 &  -0.06 & 2 & 0.45 & 5 & 1\\ 
\hline
Mean & 24.18 & 25.03 & 51.80 & 24.39 & 25.17 & 80.21 & 1.03 & & 0.55 \\
Total & & & & & & & & 71 & & 73 & 18\\
   \hline
\end{tabular}
\label{tab:cbrkga_rand_rls}
\end{table}

\begin{table}[ht]
\centering
\small
\caption{BRKGA-B and BRKGA+R+LS results for the geometric graphs.}
\begin{tabular}{l|rrr|rrrrrrrr}
  \hline
\multicolumn{1}{c|}{} & 
\multicolumn{3}{c|}{BRKGA-B} & 
\multicolumn{6}{c}{BRKGA+R+LS} \\
group & mean & max & ttb & mean & max & ttb & $diff_{m}$ & $\geq_{mean}$ & $diff_{x}$ & $\geq_{max}$ & $>_{max}$\\
\hline
geo\_50\_0.2 &  8.40 &  8.40 &  0.0 &  8.40 &  8.40 &  0.0 &  0.00 & 5 &  0.00 & 5 & 0 \\ 
geo\_50\_0.4 & 22.01 & 22.20 & 32.1 & 22.06 & 22.20 & 30.3 & 0.21 & 5 &  0.00 & 5 & 0\\ 
geo\_50\_0.6 & 30.95 & 31.20 & 27.1 & 31.01 & 31.20 & 29.9 & 0.17 & 4 &  0.00 & 5 & 0\\ 
geo\_50\_0.8 & 37.79 & 37.80 &  2.2 & 37.80 & 37.80 &  3.8 & 0.02 & 5 &  0.00 & 5 & 0\\ 
geo\_60\_0.2 &  9.60 &  9.80 & 11.1 &  9.80 &  9.80 & 12.7 & 2.14 & 5 &  0.00 & 5 & 0\\ 
geo\_60\_0.4 & 26.16 & 26.80 & 51.0 & 26.19 & 27.00 & 43.5 & 0.13 & 4 & 0.71 & 5 & 1\\ 
geo\_60\_0.6 & 37.45 & 38.00 & 28.1 & 37.46 & 38.00 & 42.3 & 0.04 & 4 &  0.00 & 5 & 0\\ 
geo\_60\_0.8 & 47.34 & 47.40 &  9.8 & 47.38 & 47.40 & 20.5 & 0.09 & 5 &  0.00 & 5 & 0\\ 
geo\_70\_0.2 & 11.89 & 12.00 &  5.5 & 12.00 & 12.00 & 14.3 & 0.94 & 5 &  0.00 & 5 & 0\\ 
geo\_70\_0.4 & 26.85 & 27.80 & 48.2 & 26.99 & 28.20 & 56.7 & 0.53 & 4 & 1.54 & 5 & 2\\ 
geo\_70\_0.6 & 41.40 & 42.20 & 56.8 & 41.35 & 42.00 & 81.1 & -0.12 & 2 &  -0.45 & 4 & 0\\ 
geo\_70\_0.8 & 53.66 & 54.20 & 29.5 & 53.68 & 54.20 & 40.6 & 0.02 & 3 &  0.00 & 5 & 0\\ 
geo\_80\_0.2 & 12.00 & 12.00 &  0.1 & 12.00 & 12.00 &  0.3 &  0.00 & 5 &  0.00 & 5 & 0\\ 
geo\_80\_0.4 & 31.35 & 32.20 & 52.8 & 31.11 & 32.20 & 72.8 & -0.78 & 1 & 0.07 & 4 & 1\\ 
geo\_80\_0.6 & 47.60 & 48.40 & 45.2 & 47.37 & 48.20 & 65.9 & -0.49 & 0 &  -0.43 & 4 & 0\\ 
geo\_80\_0.8 & 62.32 & 62.40 &  8.1 & 62.30 & 62.40 & 26.0 & -0.03 & 3 &  0.00 & 5 & 0\\ 
\hline
Mean & 31.67 & 32.05 & 25.47 & 31.68 & 32.06 & 33.79 & 0.17 & & 0.09 \\
Total & & & & & & & & 60 & & 77 & 4\\
   \hline
\end{tabular}
\label{tab:cbrkga_geo_rls}
\end{table}

\begin{table}[ht]
\centering
\small
\caption{BRKGA-B and BRKGA+R+LS results for the complement of bipartite graphs.}
\begin{tabular}{l|rrr|rrrrrrrr}
  \hline
\multicolumn{1}{c|}{} & 
\multicolumn{3}{c|}{BRKGA-B} & 
\multicolumn{6}{c}{BRKGA+R+LS} \\
group & mean & max & ttb & mean & max & ttb & $diff_{m}$ & $\geq_{mean}$ & $diff_{x}$ & $\geq_{max}$ & $>_{max}$ \\
\hline
cbip\_50\_0.2 & 34.54 & 34.60 & 12.5 & 34.60 & 34.60 & 10.3 & 0.17 & 5 & 0.00 & 5 & 0\\ 
cbip\_50\_0.4 & 31.03 & 31.20 &  3.5 & 31.16 & 31.20 & 18.5 & 0.43 & 5 & 0.00 & 5 & 0\\ 
cbip\_50\_0.6 & 29.18 & 29.20 &  1.4 & 29.20 & 29.20 &  2.9 & 0.07 & 5 & 0.00 & 5 & 0\\ 
cbip\_50\_0.8 & 30.40 & 30.40 &  2.4 & 30.40 & 30.40 &  0.8 &  0.00 & 5 & 0.00 & 5 & 0\\ 
cbip\_60\_0.2 & 41.13 & 41.40 &  4.9 & 41.25 & 41.40 & 22.1 & 0.29 & 5 & 0.00 & 5 & 0\\ 
cbip\_60\_0.4 & 37.82 & 38.00 &  7.9 & 37.84 & 38.00 & 17.2 & 0.05 & 4 & 0.00 & 5 & 0\\ 
cbip\_60\_0.6 & 35.26 & 35.60 &  3.3 & 35.41 & 35.60 & 30.7 & 0.45 & 5 & 0.00 & 5 & 0\\ 
cbip\_60\_0.8 & 34.51 & 34.60 & 18.5 & 34.58 & 34.60 & 17.6 & 0.21 & 5 & 0.00 & 5 & 0\\ 
cbip\_70\_0.2 & 46.81 & 47.20 &  2.1 & 46.94 & 47.20 & 22.9 & 0.28 & 5 & 0.00 & 5 & 0\\ 
cbip\_70\_0.4 & 44.46 & 44.80 & 14.4 & 44.42 & 44.80 & 25.4 & -0.09 & 2 & 0.00 & 5 & 0\\ 
cbip\_70\_0.6 & 41.72 & 41.80 & 17.3 & 41.72 & 41.80 & 26.1 &  0.00 & 4 & 0.00 & 5 & 0\\ 
cbip\_70\_0.8 & 41.16 & 41.20 &  1.5 & 41.17 & 41.20 &  9.5 & 0.02 & 5 & 0.00 & 5 & 0\\ 
cbip\_80\_0.2 & 51.18 & 51.60 &  4.7 & 51.17 & 51.60 & 29.8 & -0.02 & 2 & 0.00 & 5 & 0\\ 
cbip\_80\_0.4 & 47.17 & 47.40 &  5.1 & 47.11 & 47.40 & 23.5 & -0.13 & 0 & 0.00 & 5 & 0\\ 
cbip\_80\_0.6 & 46.21 & 46.40 &  0.9 & 46.21 & 46.40 &  7.3 &  0.00 & 2 & 0.00 & 5 & 0\\ 
cbip\_80\_0.8 & 44.73 & 44.80 & 18.9 & 44.72 & 44.80 & 18.3 & -0.03 & 4 & 0.00 & 5 & 0\\ 
\hline
Mean & 39.83 & 40.01 & 7.45 & 39.86 & 40.01 & 17.68 & 0.10 & & 0.00 \\
Total & & & & & & & & 63 & & 80 & 0\\
\hline
\end{tabular}
\label{tab:cbrkga_cbip_rls}
\end{table}

\begin{table}[ht]
\small
\centering
\caption{BRKGA-B and BRKGA+R+LS results for the DIMACS graphs.}
\begin{tabular}{l|rrr|rrrrr}
  \hline
\multicolumn{1}{c|}{} & 
\multicolumn{3}{c|}{BRKGA-B} & 
\multicolumn{5}{c}{BRKGA+R+LS} \\
instance & mean & max & ttb & mean & max & ttb & $diff_{m}$ & $diff_{x}$ \\
 \hline
brock200\_2 &  44.14 &  46.00 & 158.8 &  45.82 &  47.00 & 165.5 & 3.81 & 2.17 \\ 
c-fat200-1 &  18.00 &  18.00 &   0.2 &  18.00 &  18.00 &   0.4 &  0.00 &  0.00 \\ 
c-fat200-2 &  34.00 &  34.00 &   1.3 &  34.00 &  34.00 &   4.3 &  0.00 &  0.00 \\ 
c-fat200-5 &  86.62 &  87.00 &   3.0 &  86.46 &  87.00 &  82.2 &  -0.18 &  0.00 \\ 
c-fat500-1 &  20.00 &  20.00 &   0.6 &  20.00 &  20.00 &  16.1 &  0.00 &  0.00 \\ 
c-fat500-2 &  38.00 &  38.00 &  28.3 &  38.00 &  38.00 &  36.7 &  0.00 &  0.00 \\ 
C125.9 &  76.78 &  78.00 &  48.4 &  75.36 &  77.00 & 148.5 & -1.85 & -1.28 \\ 
DSJC125.1 &  12.90 &  13.00 &  16.5 &  12.92 &  13.00 & 111.6 & 0.16 &  0.00 \\ 
DSJC125.5 &  36.92 &  38.00 & 106.2 &  35.92 &  37.00 & 125.1 & -2.71 & -2.63 \\ 
DSJC125.9 &  76.64 &  78.00 &  41.2 &  75.38 &  77.00 & 166.7 & -1.64 &  -1.28 \\ 
DSJC250.1 &  17.98 &  18.00 &  82.4 &  17.92 &  18.00 &  91.9 & -0.33 &  0.00 \\ 
DSJR500.1 &  20.00 &  20.00 &  11.6 &  19.96 &  20.00 &  14.3 & -0.20 &  0.00 \\ 
fpsol2.i.2 &  39.98 &  40.00 &  26.7 &  39.96 &  40.00 &  56.1 & -0.05 &  0.00 \\ 
fpsol2.i.3 &  39.86 &  40.00 &  41.4 &  39.96 &  40.00 &  63.9 & 0.25 &  0.00 \\ 
hamming6-2 &  40.00 &  40.00 &   0.2 &  40.00 &  40.00 &   1.0 &  0.00 &  0.00 \\ 
hamming6-4 &  13.80 &  15.00 &  71.0 &  14.08 &  15.00 &  57.5 & 2.03 &  0.00 \\ 
hamming8-2 & 159.12 & 160.00 & 217.6 & 147.66 & 150.00 & 313.6 & -7.20 & -6.25 \\ 
hamming8-4 &  38.20 &  39.00 &  71.5 &  38.96 &  41.00 &  43.4 & 1.99 & 5.13 \\ 
johnson8-2-4 &  12.00 &  12.00 &   4.5 &  12.00 &  12.00 &   1.5 &  0.00 &  0.00 \\ 
johnson8-4-4 &  29.04 &  31.00 & 100.7 &  29.62 &  31.00 & 125.2 & 2.00 &  0.00 \\ 
keller4 &  47.44 &  52.00 & 174.1 &  44.30 &  47.00 & 179.0 & -6.62 & -9.62 \\ 
le450\_15a &  30.72 &  31.00 & 167.0 &  31.24 &  32.00 & 137.1 & 1.69 & 3.23 \\ 
le450\_15b &  31.06 &  32.00 & 166.2 &  31.66 &  33.00 & 164.3 & 1.93 & 3.12 \\ 
le450\_25a &  43.40 &  44.00 & 176.4 &  43.06 &  44.00 & 131.1 & -0.78 &  0.00 \\ 
le450\_25b &  41.82 &  42.00 & 102.2 &  42.16 &  44.00 & 110.2 & 0.81 & 4.76 \\ 
le450\_5a &  17.90 &  19.00 & 121.3 &  17.94 &  18.00 &  67.8 & 0.22 & -5.26 \\ 
le450\_5b &  17.96 &  18.00 & 172.2 &  17.88 &  18.00 &  72.4 & -0.45 &  0.00 \\ 
le450\_5c &  21.44 &  22.00 & 134.6 &  22.14 &  23.00 & 111.5 & 3.26 & 4.55 \\ 
le450\_5d &  21.28 &  22.00 & 151.3 &  22.00 &  23.00 & 130.0 & 3.38 & 4.55 \\ 
MANN\_a9 &  21.00 &  21.00 &   0.0 &  21.00 &  21.00 &   0.0 &  0.00 &  0.00 \\ 
mulsol.i.1 &  52.00 &  52.00 &   3.2 &  52.00 &  52.00 &   7.2 &  0.00 &  0.00 \\ 
mulsol.i.2 &  33.42 &  34.00 &  44.8 &  33.68 &  35.00 &  77.4 & 0.78 & 2.94 \\ 
mulsol.i.3 &  33.56 &  34.00 &  71.5 &  33.76 &  34.00 &  87.1 & 0.60 &  0.00 \\ 
mulsol.i.4 &  33.46 &  34.00 &  57.6 &  33.66 &  34.00 &  69.9 & 0.60 &  0.00 \\ 
mulsol.i.5 &  34.00 &  34.00 &   0.4 &  34.00 &  34.00 &   2.2 &  0.00 &  0.00 \\ 
R125.1 &   7.00 &   7.00 &   0.0 &   7.00 &   7.00 &   0.0 &  0.00 &  0.00 \\ 
R125.1c &  62.00 &  62.00 &   0.8 &  62.00 &  62.00 &  17.8 &  0.00 &  0.00 \\ 
R125.5 &  64.40 &  66.00 &  47.3 &  63.44 &  65.00 &  98.9 & -1.49 & -1.52 \\ 
R250.1 &  11.10 &  12.00 &   2.7 &  11.52 &  12.00 &  45.4 & -3.78 &  0.00 \\ 
zeroin.i.1 &  52.86 &  53.00 &  84.3 &  52.96 &  54.00 &  62.6 & -0.19 & -1.89 \\ 
zeroin.i.2 &  35.04 &  37.00 &  29.8 &  34.92 &  37.00 &  69.9 &  -0.34 &  0.00 \\ 
zeroin.i.3 &  35.06 &  37.00 &  34.9 &  34.94 &  37.00 &  66.1 & -0.34 &  0.00 \\ 
\hline
   Mean & 38.14 & 38.80 & 66.0 & 37.89 & 38.59 & 79.3 & 0.07 & 0.10 \\
   \hline
\end{tabular}
\label{tab:cbrkga_dimacs_rls}
\end{table}
\clearpage
\section{BRKGA results for the Grundy coloring problem}
\label{sec:brkga_results_grundy}

The following Tables \ref{tab:brkga_rand_rls}-\ref{tab:brkga_dimacs_rls} presents the results of BRKGA+R+LS compared with the BRKGA approach proposed in \cite{Silvaetal2024}, which will be identified here as BRKGA-G. The first column represents an instance group, so all the values in each row correspond to the average over five instances, except for \textit{DIMACS} instances where each represents a single instance. Columns 2-4, 5-7 provide the average mean, maximum, and time to best (ttb) considering the 50 independent runs for each of the instance for the \textbf{BRKGA-G} (literature results) and \textbf{BRKGA+R+LS}, respectively. Columns $diff_m = 100*\frac{mean_{\textrm{BRKGA+R+LS}}-mean_{\textrm{BRKGA-G}}}{mean_{\textrm{BRKGA-G}}}$ and $diff_x = 100*\frac{max_{\textrm{BRKGA+R+LS}} - max_{\textrm{BRKGA-G}}}{max_{\textrm{BRKGA-G}}}$, and represent the percentage improvement in relation to literature results for mean and best solution. Positive values for the $diff$ columns indicate better results from $BRGKA+R+LS$. The $\geq_{mean}$ and $\geq_{max}$ columns indicate how many times the new approach obtained better or equal solutions within each group about the average and maximum colors obtained by pure BRKGA-G, respectively. The last column ($>_{max}$) reports the quantity in which it was strictly better. The tables' last two rows indicate the average values across all rows and the total number of instances that the new approach was better than or equal to the literature results. Note that Table \ref{tab:brkga_dimacs_rls} corresponding to \textit{DIMACS} instances does not contain the $\geq_{mean}$, $\geq_{max}$ columns or the \textit{Total} row because each row corresponds to a single instance.

\begin{table}[ht]
\centering
\small
\caption{BRKGA-G and BRKGA+R+LS results for the random graphs.}
\begin{tabular}{l|rrr|rrrrrrrr}
  \hline
\multicolumn{1}{c|}{} & 
\multicolumn{3}{c|}{BRKGA-G} & 
\multicolumn{6}{c}{BRKGA+R+LS} \\
group & mean & max & ttb & mean & max & ttb & $diff_{m}$ & $\geq_{mean}$ & $diff_{x}$ & $\geq_{max}$ & $>_{max}$\\
 \hline
rand\_50\_0.2 & 11.18 & 11.60 & 35.8 & 11.27 & 11.60 &  29.8 & 0.80 & 5 &  0.00 & 5 & 0\\ 
rand\_50\_0.4 & 16.26 & 17.00 & 51.7 & 16.66 & 17.20 &  74.1 & 2.41 & 5 & 1.18 & 5 & 1\\ 
rand\_50\_0.6 & 22.62 & 23.40 & 81.1 & 23.00 & 23.60 &  77.6 & 1.67 & 5 & 0.87 & 5 & 1\\ 
rand\_50\_0.8 & 30.79 & 31.40 & 77.9 & 30.88 & 31.40 &  68.2 & 0.28 & 4 &  0.00 & 5 & 0 \\ 
rand\_60\_0.2 & 12.15 & 12.60 & 46.3 & 12.38 & 12.60 &  48.9 & 1.92 & 5 &  0.00 & 5 & 0\\ 
rand\_60\_0.4 & 18.84 & 19.80 & 65.2 & 19.18 & 20.00 &  85.3 & 1.82 & 5 & 1.05 & 5 & 1\\ 
rand\_60\_0.6 & 25.66 & 26.40 & 95.3 & 25.99 & 26.80 &  91.6 & 1.28 & 5 & 1.54 & 5 & 2 \\ 
rand\_60\_0.8 & 35.86 & 36.80 & 78.9 & 36.07 & 37.00 &  93.5 & 0.59 & 5 & 0.56 & 4 & 2\\ 
rand\_70\_0.2 & 13.19 & 14.00 & 31.5 & 13.48 & 14.00 &  56.2 & 2.22 & 4 &  0.00 & 5 & 0\\ 
rand\_70\_0.4 & 20.79 & 21.80 & 79.1 & 21.10 & 21.80 &  77.9 & 1.50 & 5 &  0.00 & 5 & 0\\ 
rand\_70\_0.6 & 28.11 & 29.20 & 83.9 & 28.46 & 29.40 & 102.5 & 1.25 & 5 & 0.71 & 4 & 2\\ 
rand\_70\_0.8 & 39.54 & 40.60 & 93.9 & 39.67 & 40.60 & 108.8 & 0.31 & 4 &  0.00 & 5 & 0\\ 
rand\_80\_0.2 & 14.31 & 15.00 & 49.2 & 14.68 & 15.00 &  76.9 & 2.59 & 5 & 0.00 & 5 & 0\\ 
rand\_80\_0.4 & 22.27 & 23.40 & 66.8 & 22.55 & 23.20 &  79.5 & 1.25 & 5 & -0.83 & 4 & 0\\ 
rand\_80\_0.6 & 31.36 & 32.60 & 94.2 & 31.60 & 32.60 & 108.6 & 0.75 & 5 &  0.00 & 5 & 0\\ 
rand\_80\_0.8 & 43.14 & 44.60 & 82.1 & 43.40 & 44.60 & 123.8 & 0.59 & 5 & 0.01 & 4 & 1\\ 
\hline
Mean & 24.12 & 25.01 & 69.5 & 24.39 & 25.08 & 81.45 & 1.32 & & 0.31 \\
Total & & & & & & & & 77 & & 76 & 10\\
\hline
\end{tabular}
\label{tab:brkga_rand_rls}
\end{table}

\begin{table}[ht]
\centering
\small
\caption{BRKGA-G and BRKGA+R+LS results for the geometric graphs.}
\begin{tabular}{l|rrr|rrrrrrrr}
  \hline
\multicolumn{1}{c|}{} & 
\multicolumn{3}{c|}{BRKGA-G} & 
\multicolumn{6}{c}{BRKGA+R+LS} \\
group & mean & max & ttb & mean & max & ttb & $diff_{m}$ & $\geq_{mean}$ & $diff_{x}$ & $\geq_{max}$ & $>_{max}$\\
 \hline
geo\_50\_0.2 &  8.92 &  9.00 & 20.0 &  8.98 &  9.00 & 16.8 & 0.59 & 5 &  0.00 & 5 & 0\\ 
geo\_50\_0.4 & 22.12 & 22.20 & 26.0 & 22.12 & 22.20 & 17.7 & 0.02 & 4 &  0.00 & 5 & 0\\ 
geo\_50\_0.6 & 30.99 & 31.40 & 31.3 & 31.08 & 31.60 & 29.9 & 0.33 & 5 & 0.65 & 5 & 1\\ 
geo\_50\_0.8 & 37.82 & 38.00 &  3.9 & 37.82 & 38.00 &  4.8 &  0.00 & 5 &  0.00 & 5 & 0\\ 
geo\_60\_0.2 & 10.00 & 10.00 &  5.2 & 10.00 & 10.00 &  0.9 &  0.00 & 5 &  0.00 & 5 & 0\\ 
geo\_60\_0.4 & 26.25 & 26.60 & 41.7 & 26.34 & 26.80 & 41.5 & 0.35 & 5 & 0.71 & 5 & 1 \\ 
geo\_60\_0.6 & 37.52 & 38.00 & 39.0 & 37.64 & 38.00 & 45.0 & 0.30 & 5 &  0.00 & 5 & 0\\ 
geo\_60\_0.8 & 47.47 & 47.60 & 19.8 & 47.52 & 47.60 & 24.3 & 0.12 & 4 &  0.00 & 5 & 0\\ 
geo\_70\_0.2 & 12.24 & 12.40 &  5.3 & 12.34 & 12.40 & 21.7 & 0.86 & 5 &  0.00 & 5 & 0\\ 
geo\_70\_0.4 & 27.43 & 28.40 & 55.9 & 27.66 & 28.40 & 60.3 & 0.81 & 4 &  -0.03 & 4 & 1 \\ 
geo\_70\_0.6 & 41.70 & 42.40 & 42.1 & 41.78 & 42.60 & 47.6 & 0.18 & 4 & 0.45 & 5 & 1\\ 
geo\_70\_0.8 & 53.98 & 54.20 & 36.3 & 54.00 & 54.20 & 45.1 & 0.03 & 3 &  0.00 & 5 & 0\\ 
geo\_80\_0.2 & 12.00 & 12.00 &  0.1 & 12.00 & 12.00 &  0.1 &  0.00 & 5 &  0.00 & 5 & 0\\ 
geo\_80\_0.4 & 31.49 & 32.20 & 61.0 & 31.54 & 32.40 & 61.1 & 0.13 & 4 & 0.62 & 5 & 1\\ 
geo\_80\_0.6 & 47.72 & 48.40 & 52.1 & 47.74 & 48.60 & 69.9 & 0.05 & 4 & 0.43 & 5 & 1\\ 
geo\_80\_0.8 & 62.72 & 63.00 & 24.7 & 62.74 & 63.00 & 46.4 & 0.03 & 4 &  0.00 & 5 & 0 \\ 
\hline
Mean & 31.92 & 32.26 & 29.05 & 31.95 & 32.30 & 33.31 & 0.23 & & 0.17 \\
Total & & & & & & & & 71 & & 79 & 6\\
\hline
\end{tabular}
\label{tab:brkga_geo_rls}
\end{table}

\begin{table}[ht]
\centering
\small
\caption{BRKGA-G and BRKGA+R+LS results for the bipartite graphs.}
\begin{tabular}{l|rrr|rrrrrrrr}
  \hline
\multicolumn{1}{c|}{} & 
\multicolumn{3}{c|}{BRKGA-G} & 
\multicolumn{6}{c}{BRKGA+R+LS} \\
group & mean & max & ttb & mean & max & ttb & $diff_{m}$ & $\geq_{mean}$ & $diff_{x}$ & $\geq_{max}$ & $>_{max}$\\
\hline
bip\_50\_0.2 &  7.56 &  8.00 &  31.2 &  7.74 &  8.00 &  44.2 & 2.46 & 5 &  0.00 & 5 & 0\\ 
bip\_50\_0.4 & 10.05 & 11.00 &  19.6 & 10.15 & 11.00 &  23.3 & 0.99 & 5 &  0.00 & 5 & 0\\ 
bip\_50\_0.6 & 12.76 & 13.40 &  65.1 & 12.92 & 13.60 &  46.1 & 1.22 & 5 & 1.54 & 5 & 1\\ 
bip\_50\_0.8 & 14.93 & 16.00 &  59.4 & 15.44 & 16.20 &  80.0 & 3.45 & 5 & 1.43 & 5 & 1\\ 
bip\_60\_0.2 &  8.15 &  8.60 &  19.4 &  8.41 &  8.80 &  47.1 & 3.15 & 5 & 2.50 & 5 & 1\\ 
bip\_60\_0.4 & 10.90 & 11.60 &  53.3 & 11.08 & 11.80 &  37.4 & 1.63 & 5 & 1.82 & 5 & 1\\ 
bip\_60\_0.6 & 13.86 & 15.00 &  82.5 & 14.12 & 15.00 &  74.7 & 1.88 & 5 &  0.00 & 5 & 0\\ 
bip\_60\_0.8 & 17.17 & 18.80 &  92.7 & 17.65 & 18.80 & 102.3 & 2.84 & 5 & 0.11 & 4 & 1\\ 
bip\_70\_0.2 &  8.65 &  9.20 &  27.7 &  8.87 &  9.20 &  45.8 & 2.58 & 5 &  0.00 & 5 & 0\\ 
bip\_70\_0.4 & 11.67 & 12.40 &  62.4 & 11.90 & 12.60 &  52.2 & 1.96 & 5 & 1.67 & 5 & 1\\ 
bip\_70\_0.6 & 14.63 & 15.80 &  81.9 & 14.99 & 16.00 &  83.0 & 2.46 & 5 & 1.33 & 5 & 1\\ 
bip\_70\_0.8 & 18.36 & 20.20 & 110.7 & 18.95 & 20.40 & 111.6 & 3.25 & 5 & 1.05 & 5 & 1\\ 
bip\_80\_0.2 &  9.36 &  9.80 &  49.1 &  9.64 &  9.80 &  62.0 & 3.05 & 5 &  0.00 & 5 & 0\\ 
bip\_80\_0.4 & 12.84 & 13.80 &  76.3 & 13.05 & 14.00 &  67.3 & 1.60 & 5 & 1.54 & 5 & 1\\ 
bip\_80\_0.6 & 16.00 & 17.00 & 108.6 & 16.34 & 17.20 &  91.5 & 2.18 & 5 & 1.18 & 5 & 1\\ 
bip\_80\_0.8 & 19.38 & 21.20 & 115.7 & 20.06 & 21.60 & 125.7 & 3.52 & 5 & 1.90 & 5 & 2\\ 
\hline
Mean & 12.89 & 13.86 & 65.97 & 13.20 & 14.00 & 68.32 & 2.38 & & 1.04 \\
Total & & & & & & & & 80 & & 79 & 12\\
   \hline
\end{tabular}
\label{tab:brkga_bip_rls}
\end{table}

\begin{table}[ht]
\centering
\small
\caption{BRKGA-G and BRKGA+R+LS results for the complement of bipartite graphs.}
\begin{tabular}{l|rrr|rrrrrrrr}
  \hline
\multicolumn{1}{c|}{} & 
\multicolumn{3}{c|}{BRKGA-G} & 
\multicolumn{6}{c}{BRKGA+R+LS} \\
group & mean & max & ttb & mean & max & ttb & $diff_{m}$ & $\geq_{mean}$ & $diff_{x}$ & $\geq_{max}$ & $>_{max}$\\
\hline
cbip\_50\_0.2 & 34.52 & 34.60 & 16.6 & 34.60 & 34.60 & 13.8 & 0.25 & 5 & 0.00 & 5 & 0\\ 
cbip\_50\_0.4 & 31.04 & 31.20 &  4.8 & 31.18 & 31.20 & 10.3 & 0.42 & 5 & 0.00 & 5 & 0\\ 
cbip\_50\_0.6 & 29.16 & 29.20 &  3.9 & 29.20 & 29.20 &  2.7 & 0.13 & 5 & 0.00 & 5 & 0\\ 
cbip\_50\_0.8 & 30.40 & 30.40 &  1.4 & 30.40 & 30.40 &  0.9 &  0.00 & 5 & 0.00 & 5 & 0\\ 
cbip\_60\_0.2 & 41.27 & 41.40 &  9.2 & 41.37 & 41.40 & 20.5 & 0.25 & 4 & 0.00 & 5 & 0\\ 
cbip\_60\_0.4 & 37.84 & 38.20 & 11.4 & 37.94 & 38.20 & 21.2 & 0.27 & 4 & 0.00 & 5 & 0\\ 
cbip\_60\_0.6 & 35.34 & 35.80 &  6.3 & 35.52 & 35.80 & 32.6 & 0.54 & 5 & 0.00 & 5 & 0 \\ 
cbip\_60\_0.8 & 34.53 & 34.60 & 19.3 & 34.59 & 34.60 & 16.8 & 0.18 & 5 & 0.00 & 5 & 0 \\ 
cbip\_70\_0.2 & 46.77 & 47.20 &  9.6 & 46.89 & 47.20 & 26.3 & 0.26 & 5 & 0.00 & 5 & 0\\ 
cbip\_70\_0.4 & 44.49 & 44.80 &  8.6 & 44.52 & 44.80 & 19.0 & 0.07 & 4 & 0.00 & 5 & 0 \\ 
cbip\_70\_0.6 & 41.67 & 41.80 & 16.6 & 41.72 & 41.80 & 29.0 & 0.12 & 5 & 0.00 & 5 & 0 \\ 
cbip\_70\_0.8 & 41.12 & 41.20 &  1.4 & 41.17 & 41.20 &  6.9 & 0.12 & 5 & 0.00 & 5 & 0 \\ 
cbip\_80\_0.2 & 51.13 & 51.60 &  3.5 & 51.13 & 51.60 & 24.6 & 0.01 & 4 & 0.00 & 5 & 0 \\ 
cbip\_80\_0.4 & 47.12 & 47.40 &  7.9 & 47.12 & 47.40 & 22.6 & 0.01 & 3 & 0.00 & 5 & 0 \\ 
cbip\_80\_0.6 & 46.18 & 46.40 &  1.1 & 46.24 & 46.40 & 13.3 & 0.13 & 5 & 0.00 & 5 & 0 \\ 
cbip\_80\_0.8 & 44.75 & 44.80 & 19.8 & 44.74 & 44.80 & 17.1 & -0.03 & 4 & 0.00 & 5 & 0 \\ 
\hline
Mean & 39.83 & 40.03 & 8.83 & 39.89 & 40.03 & 17.35 & 0.17 & & 0.00 \\
Total & & & & & & & & 73 & & 80 & 0\\
\hline
\end{tabular}
\label{tab:brkga_cbip_rls}
\end{table}

\begin{table}[ht]
\centering
\small
\caption{BRKGA-G and BRKGA+R+LS results for the DIMACS graphs.}
\begin{tabular}{l|rrr|rrrrr}
  \hline
\multicolumn{1}{c|}{} & 
\multicolumn{3}{c|}{BRKGA-G} & 
\multicolumn{5}{c}{BRKGA+R+LS} \\
instance & mean & max & ttb & mean & max & ttb & $diff_{m}$ & $diff_{x}$ \\
 \hline
brock200\_2 &  44.24 &  45.00 & 145.4 &  46.12 &  48.00 & 180.9 & 4.25 & 6.67 \\ 
c-fat200-1 &  18.00 &  18.00 &   0.2 &  18.00 &  18.00 &   0.4 &  0.00 &  0.00 \\ 
c-fat200-2 &  35.00 &  35.00 &   1.4 &  35.00 &  35.00 &   5.6 &  0.00 &  0.00 \\ 
c-fat200-5 &  86.80 &  87.00 &   2.6 &  86.58 &  87.00 &  81.8 & -0.25 &  0.00 \\ 
c-fat500-1 &  21.00 &  21.00 &   4.1 &  21.00 &  21.00 &   5.5 &  0.00 &  0.00 \\ 
c-fat500-2 &  39.00 &  39.00 &  39.7 &  39.00 &  39.00 &  45.2 &  0.00 &  0.00 \\ 
C125.9 &  76.82 &  78.00 &  70.5 &  75.68 &  77.00 & 161.6 &  -1.48 &  -1.28 \\ 
DSJC125.1 &  12.60 &  13.00 &  29.4 &  13.00 &  13.00 &  69.6 & 3.17 &  0.00 \\ 
DSJC125.5 &  35.94 &  38.00 &  47.3 &  36.02 &  38.00 & 143.9 & 0.22 &  0.00 \\ 
DSJC125.9 &  76.50 &  78.00 &  58.5 &  74.98 &  77.00 & 156.7 &  -1.99 &  -1.28 \\ 
DSJC250.1 &  18.00 &  18.00 &  53.9 &  18.02 &  19.00 &  83.4 & 0.11 & 5.56 \\ 
DSJR500.1 &  19.98 &  20.00 & 7.1 &  19.98 &  20.00 &  14.1 &  0.00 &  0.00 \\ 
fpsol2.i.2 &  39.98 &  40.00 &  36.8 &  40.00 &  40.00 &  45.5 & 0.05 &  0.00 \\ 
fpsol2.i.3 &  40.00 &  40.00 &  48.6 &  39.94 &  40.00 &  65.5 & -0.15 &  0.00 \\ 
hamming6-2 &  40.00 &  40.00 &   0.2 &  40.00 &  40.00 &   1.0 &  0.00 &  0.00 \\ 
hamming6-4 &  13.96 &  15.00 &  57.4 &  14.10 &  15.00 &  48.2 & 1.00 &  0.00 \\ 
hamming8-2 & 159.86 & 161.00 & 199.9 & 147.56 & 151.00 & 312.3 & -7.69 & -6.21 \\ 
hamming8-4 &  38.20 &  39.00 &  80.5 &  39.06 &  42.00 &  47.6 & 2.25 & 7.69 \\ 
johnson8-2-4 &  12.00 &  12.00 &   6.1 &  12.00 &  12.00 &   1.2 &  0.00 &  0.00 \\ 
johnson8-4-4 &  29.22 &  31.00 & 116.7 &  29.76 &  31.00 & 115.6 & 1.85 &  0.00 \\ 
keller4 &  45.50 &  48.00 & 140.4 &  44.22 &  47.00 & 199.2 &  -2.81 &  -2.08 \\ 
le450\_15a &  31.27 &  32.00 & 157.5 &  31.32 &  32.00 & 132.4 & 0.17 &  0.00 \\ 
le450\_15b &  31.82 &  32.00 & 167.9 &  32.06 &  33.00 & 145.0 & 0.75 & 3.12 \\ 
le450\_25a &  43.30 &  44.00 & 135.9 &  43.24 &  44.00 & 154.1 &  -0.14 &  0.00 \\ 
le450\_25b &  42.12 &  43.00 & 124.1 &  42.30 &  44.00 &  97.4 & 0.43 & 2.33 \\ 
le450\_5a &  18.00 &  19.00 & 119.4 &  17.98 &  18.00 &  79.8 &  -0.11 &  -5.26 \\ 
le450\_5b &  17.98 &  18.00 & 112.1 &  18.02 &  19.00 &  76.8 & 0.22 & 5.56 \\ 
le450\_5c &  22.09 &  23.00 & 158.2 &  22.08 &  23.00 & 102.2 &  -0.06 &  0.00 \\ 
le450\_5d &  21.94 &  22.00 & 172.0 &  22.04 &  23.00 &  92.0 & 0.46 & 4.55 \\ 
MANN\_a9 &  21.00 &  21.00 &   0.0 &  21.00 &  21.00 &   0.0 &  0.00 &  0.00 \\ 
mulsol.i.1 &  52.00 &  52.00 &   4.4 &  52.00 &  52.00 &   7.2 &  0.00 &  0.00 \\ 
mulsol.i.2 &  33.54 &  34.00 &  69.0 &  33.70 &  34.00 &  57.8 & 0.48 &  0.00 \\ 
mulsol.i.3 &  33.56 &  34.00 &  41.6 &  33.78 &  34.00 &  80.3 & 0.67 &  0.00 \\ 
mulsol.i.4 &  33.50 &  34.00 &  65.4 &  33.72 &  34.00 &  76.7 & 0.66 &  0.00 \\ 
mulsol.i.5 &  34.00 &  34.00 &   0.4 &  34.00 &  34.00 &   2.5 &  0.00 &  0.00 \\ 
R125.1 &   7.00 &   7.00 &   0.0 &   7.00 &   7.00 &   0.0 &  0.00 &  0.00 \\ 
R125.1c &  62.00 &  62.00 &   0.8 &  62.00 &  62.00 &  17.8 &  0.00 &  0.00 \\ 
R125.5 &  63.84 &  65.00 &  41.6 &  63.40 &  65.00 & 104.2 &  -0.69 &  0.00 \\ 
R250.1 &  12.00 &  12.00 &  16.5 &  12.00 &  12.00 &  12.4 &  0.00 &  0.00 \\ 
zeroin.i.1 &  52.88 &  53.00 &  88.0 &  52.98 &  54.00 &  84.8 & 0.19 & 1.89 \\ 
zeroin.i.2 &  35.06 &  37.00 &  55.0 &  35.10 &  37.00 &  62.5 & 0.11 &  0.00 \\ 
zeroin.i.3 &  35.04 &  37.00 &  45.8 &  35.06 &  37.00 &  63.2 & 0.06 &  0.00 \\ 
\hline
Mean & 38.2 & 38.8 & 64.8 & 37.9 & 38.7 & 76.9 & 0.04 & 0.50 \\
\hline
\end{tabular}
\label{tab:brkga_dimacs_rls}
\end{table}
\clearpage
\section{Summary of the DIMACS' instances characteristics}
\label{sec:dimacs_instance}

\begin{table}[ht]
\centering
\caption{Characteristics of the DIMACS instances}
\footnotesize
\begin{tabular}{lrr}
\hline
instance & $|V|$ & $d(G)$ \\ 
  \hline
johnson8-2-4 & 28 & 0.55  \\
johnson8-4-4 & 70 & 0.78  \\
mann\_a9 & 45 & 0.92  \\
hamming6-2 & 64 & 0.90  \\
hamming6-4 & 64 & 0.34  \\
c125.9 & 125 & 0.89  \\
dsjc125.1 & 125 & 0.09  \\
dsjc125.5 & 125 & 0.50  \\
dsjc125.9 & 125 & 0.89  \\
r125.1 & 125 & 0.02  \\
r125.1c & 125 & 0.96  \\
r125.5 & 125 & 0.49  \\
keller4 & 147 & 0.64  \\
mulsol.i.1 & 197 & 0.20  \\
mulsol.i.2 & 188 & 0.22  \\
mulsol.i.3 & 184 & 0.23  \\
mulsol.i.4 & 185 & 0.23  \\
mulsol.i.5 & 186 & 0.23  \\
brock200\_2 & 200 & 0.49  \\
c-fat200-1 & 200 & 0.07  \\
c-fat200-2 & 200 & 0.16  \\
    \hline
\end{tabular}
\hspace{1em}
\begin{tabular}{lrr}
\hline
instance & $|V|$ & $d(G)$ \\ 
  \hline
    
    c-fat200-5 & 200 & 0.42  \\
    zeroin.i.1 & 211 & 0.18  \\
    zeroin.i.2 & 206 & 0.15  \\
    zeroin.i.3 & 206 & 0.16  \\
    dsjc250.1 & 250 & 0.10  \\
    r250.1 & 250 & 0.02  \\
    hamming8-2 & 256 & 0.96  \\
    hamming8-4 & 256 & 0.63  \\
    fpsol2.i.2 & 451 & 0.08  \\
    fpsol2.i.3 & 425 & 0.09  \\
    le450\_5a & 450 & 0.05  \\
    le450\_5b & 450 & 0.05  \\
    le450\_5c & 450 & 0.09  \\
    le450\_5d & 450 & 0.09  \\
    le450\_15a & 450 & 0.08  \\
    le450\_15b & 450 & 0.08  \\
    le450\_25a & 450 & 0.08  \\
    le450\_25b & 450 & 0.08  \\
    dsjr500.1 & 500 & 0.02  \\
    c-fat500-1 & 500 & 0.03  \\
    c-fat500-2 & 500 & 0.07 \\
    \hline
\end{tabular}
\label{tab:dimacs_instances}
\end{table}
\end{appendices}



\end{document}